\documentclass[letterpaper, reqno, final]{amsart}

\pdfoutput=1 

\usepackage{amsmath, amssymb, amsthm}
\usepackage{dsfont,mathrsfs}

\usepackage[hidelinks]{hyperref}

\usepackage[T1]{fontenc}

\usepackage{graphicx}
\usepackage{cite}
\usepackage{enumerate}
\usepackage{bookmark}
\usepackage{mathtools}
\usepackage{cleveref}

\crefformat{equation}{(#2#1#3)}
\crefrangeformat{equation}{(#3#1#4)--(#5#2#6)}
\crefmultiformat{equation}{(#2#1#3)}%
{ and~(#2#1#3)}{, (#2#1#3)}{ and~(#2#1#3)}

\crefrangeformat{thm}{Theorems #3#1#4 to #5#2#6}
\crefmultiformat{thm}{Theorems #2#1#3}%
{ and~#2#1#3}{, #2#1#3}{ and~#2#1#3}

\calclayout

\usepackage[right]{showlabels}

\DeclareMathOperator{\supp}{supp}

\newcommand{\jBra}[1]{\langle #1 \rangle}

\newcommand{\bbR}{\mathbb{R}}

\newcommand{\bbOne}{\mathds{1}}
\newcommand{\fR}{\mathfrak{R}}

\newcommand{\cF}{\mathcal{F}}
\newcommand{\DFT}{\widetilde{\mathcal{F}}}
\newcommand{\frR}{\mathfrak{R}}

\newcommand{\waveop}{\Omega}

\newcommand{\rmI}{\mathrm{I}}
\newcommand{\rmII}{\mathrm{II}}

\newtheorem{thm}{Theorem}
\newtheorem{defn}{Definition}
\newtheorem{prop}[thm]{Proposition}
\newtheorem{lemma}[thm]{Lemma}
\newtheorem{cor}[thm]{Corollary}
\newtheorem{rmk}{Remark}

\begin{document}
\title[Asymptotics for NLS with slowly decaying potential]{Asymptotics for the cubic 1D NLS with a slowly decaying potential}

\author[Gavin Stewart]{Gavin Stewart}
 \address[Gavin Stewart]{\newline
        Department of Mathematics, \newline
         Rutgers University, New Brunswick, NJ 08903 USA.}
  \email[]{gavin.stewart@rutgers.edu}

\date{\today}
\begin{abstract}
    We consider the asymptotics of the one-dimensional cubic nonlinear Schr\"odinger equation with an external potential $V$ that does not admit bound states.  Assuming that $\jBra{x}^{2+}V(x) \in L^1$ and that $u$ is orthogonal to any zero-energy resonances, we show that solutions exhibit modified scattering behavior.  Our decay assumptions are weaker than those appearing in previous work and are likely nearly optimal for exceptional potentials, since the M\o{}ller wave operators are not known to be bounded on $L^p$ spaces unless $\jBra{x}^2V \in L^1$.  

    Our method combines two techniques: First, we prove a new class of linear estimates, which allow us to relate vector fields associated with the equation with potential to vector fields associated with the flat equation, up to localized error terms.  By using the improved local decay coming from the zero-energy assumption, we can treat these errors perturbatively, allowing us to treat the weighted estimates using the same vector field arguments as in the flat case.  Second, we derive the asymptotic profile using the wave packet testing method introduced by Ifrim and Tataru.  Again, by taking advantage of improved local decay near the origin, we are able to treat the potential as a lower-order perturbation, allowing us to use tools developed to study the flat problem.
\end{abstract}

\maketitle

\section{Introduction}

We will consider the long-time dynamics of small solutions to the nonlinear Schr\"odinger equation
\begin{equation}\label{eqn:main-eqn}
    i\partial_t u -\Delta u + V(x) u = \pm |u|^2 u
\end{equation}
on $\bbR_x \times \bbR_t$.  This equation is a prototypical model for a nonlinear dispersive equation with an external potential, and has been studied extensively: see~\cite{frohlichSolitaryWaveDynamics2004,bronskiSolitonDynamicsPotential2000,holmerSolitonInteractionSlowly2008,sofferSelectionGroundState2004,germainNonlinearSchrodingerEquation2018,germainNonlinearResonancesPotential2015,chen1dimensionalNonlinearSchrodinger2022,chen1dCubicNLS2024,delortModifiedScatteringOdd2016,naumkinSharpAsymptoticBehavior2016,chenLongtimeDynamicsSmall2023a,pusateriBilinearEstimatesPresence2024a}.  

Our main goal in this paper is to prove modified scattering asymptotics for~\eqref{eqn:main-eqn} for potentials decaying slowly at infinity.  If we want solutions to~\eqref{eqn:main-eqn} to disperse, the linear operator $H = -\Delta + V$ cannot have bound states, so we will always make the spectral assumption
\begin{equation}
    H = -\Delta + V \text{ has no eigenvalues}\label{eqn:no-bd-states-assumption}
\end{equation}
We also assume that $V$ satisfies the decay assumption:
\begin{equation}\label{eqn:V-decay-assumption}
    V \in L^1_\gamma \text{ with }\gamma > 2
\end{equation}
where $L^1_\gamma = \{V \in L^1: \jBra{x}^\gamma V \in L^1\}$.  This is an improvement on the assumption $\gamma > 5/2$ in previous works~\cite{chen1dimensionalNonlinearSchrodinger2022,chen1dCubicNLS2024}, and may be essentially optimal for exceptional potentials in the scale of weighted $L^1$ spaces, since $L^p$ boundedness of wave operators is not known for $\gamma < 2$~\cite{danconaLpBoundednessWaveOperator2006}.  Finally, we assume that the distorted Fourier transform of $u$ at $k = 0$ vanishes for all times:
\begin{equation}\label{eqn:zero-freq-hypo}
    \tilde{u}(0,t) \equiv 0
\end{equation}
which is equivalent to assuming that $u$ is orthogonal to any zero energy resonance.  Note that~\eqref{eqn:zero-freq-hypo} is automatically satisfied if $V$ is generic (see~\Cref{def:generic-def}), and will be satisfied if $V$ is an even potential with an even zero energy resonance and the initial data for~\eqref{eqn:main-eqn} is odd. 

A central difficulty in proving asymptotics for equations with potentials is to understand the space-time resonance structure of the nonlinearity.  In the flat case ($V = 0$), the eigenfunctions of $H_0 = -\Delta$ are the complex exponentials $e^{ikx}$, which have the product property
\begin{equation*}
    e^{ik_1x}e^{-ik_2x}e^{ik_3x} = e^{i(k_1-k_2+k_3)x}
\end{equation*}
In particular, this means that the interaction of three waves with frequencies $k_1$, $k_2$ and $k_3$ through the cubic nonlinearity will produce an output wave with frequency $k = k_1 - k_2 + k_3$.  Once a potential is introduced, however, this exact correspondence between input and output frequencies is lost, which leads to major difficulties in even defining resonances.  It is only in the past decade that techniques for addressing this difficulty have been proposed.  We will briefly review these methods in~\Cref{sec:prev-work}.  A common thread through all methods is the following idea: \textit{The nonlinear interactions in~\eqref{eqn:main-eqn} resemble those of the flat problem (up to manageable errors)}.

We believe that our method embodies this philosophy with minimal technical complexity.  In fact, after proving the appropriate linear estimates, we are able to reduce all nonlinear estimates for~\eqref{eqn:main-eqn} to nonlinear estimates in the flat case, up to errors which are localized in space and can thus be controlled using improved local decay given by~\Cref{lem:improved-local-decay}.  Thus, we can completely sidestep the difficulties associated with multilinear estimates in distorted frequency space at the minor cost of proving certain linear estimates.  We believe that this has the potential to simplify many arguments in the field.

\subsection{Previous work}\label{sec:prev-work}

An early result on long-time asymptotics for~\eqref{eqn:main-eqn} with nonlinearities $|u|^pu$, $p > 3$ was proved by Cuccagna, Visciglia, and Georgiev, who showed that the Klainerman-Sobolev decay estimate can be modified to cover the non-flat case $V \neq 0$~\cite{cuccagnaDecayScatteringSmall2014}.  The assumption $p > 3$ is essential here, as the Klainerman-Sobolev estimate is known to be insufficient for $p = 3$ even in the flat case.  

In dimension $3$, L\'eger was able to prove global existence and scattering for quadratic nonlinear Schr\"odinger equations with small potentials~\cite{legerGlobalExistenceScattering2021}.  His technique is based on combining ideas from space-time resonances with a Born series expansion for the potential.  This has the advantage of generalizing to equations with small electromagnetic potentials~\cite{leger3DQuadraticNLS2020}.  However, in order for the Born series to converge, the potential must be small, which is not always true in applications.

The first technique to take advantage of the resonant structure in the equation with arbitrarily large potential was introduced by Germain, Hani, and Walsh in~\cite{germainNonlinearResonancesPotential2015} for three-dimensional NLS with nonlinearity $\overline{u}^2$ and later extended to the one-dimensional cubic equation with generic potential by Germain, Pusateri, and Rousset~\cite{germainNonlinearSchrodingerEquation2018}.  Since this technique has become the most prominent in the literature, we will briefly outline the key ideas in the one-dimensional case.  The main idea is to look at the \textit{nonlinear spectral measure}, which for the cubic nonlinearity in~\eqref{eqn:main-eqn} is given by
\begin{equation*}
    \mu(k,k_1,k_2,k_3) = \int \overline{\psi(x,k)} \psi(x, k_1)\overline{\psi(x,k_2)} \psi(x, k_3)\;dx
\end{equation*}
where $\psi(x,k)$ is a solution to
\begin{equation*}
    H \psi = k^2 \psi
\end{equation*}
with appropriate boundary conditions: see~\Cref{sec:spectral-theory} for a precise definition.  For the flat problem, $\psi(x,k) = (2\pi)^{-1/2} e^{ikx}$, and we have
\begin{equation*}
    \mu_\text{flat}(k, k_1, k_2, k_3) = \frac{1}{2\pi} \delta(k - k_1 + k_2 - k_3)
\end{equation*}
In general, $\mu$ is singular near $k = k_1 - k_2 + k_3$, and more regular elsewhere.  Thus, we can decompose the nonlinear spectral measure into a singular part $\mu_S$ and a regular part $\mu_R$.\footnote{In some situations, it can also be helpful to allow for an additional term $\mu_L$ that has better behavior at low frequencies, see~\cite{germainNonlinearSchrodingerEquation2018,chen1dCubicNLS2024}.  Since we assume that $\tilde{u}(0,t) = 0$, we will ignore this low-frequency improved term.}  The singular part has a structure very similar to the flat case, which allows the authors to obtain weighted estimates by modifying the arguments in~\cite{katoNewProofLongrange2011}, while the regular part is highly localized, allowing the authors to take advantage of improved local decay for these terms.  The final modified scattering result then follows after a careful stationary phase analysis of the nonlinearity in distorted Fourier space.  These techniques were refined by Chen and Pusateri in~\cite{chen1dimensionalNonlinearSchrodinger2022}, which compared to~\cite{germainNonlinearSchrodingerEquation2018} assumes much less on the potential and also simplifies the arguments for the weighted estimates.  A further extension in~\cite{chen1dCubicNLS2024} proves an analogous result for exceptional potentials with $\gamma > 5/2$ but with the assumption~\eqref{eqn:zero-freq-hypo} replaced by a much weaker assumption on the symmetry of the zero-energy resonances of $H$.

Alternative methods have also been introduced to prove modified scattering for~\eqref{eqn:main-eqn}.  In~\cite{naumkinSharpAsymptoticBehavior2016} and~\cite{naumkinNonlinearSchrodingerEquations2018a} Naumkin adapted the operator factorization technique introduced by Hayashi and Naumkin in~\cite{hayashiAsymptoticsLargeTime1998} to obtain asymptotics for~\eqref{eqn:main-eqn}.  Masaki, Murphy and Segata also used the operator factorization technique to show modified scattering for the equation with a repulsive $\delta$ potential~\cite{masakiModifiedScatteringOneDimensional2019}.   Delort has also proved modified scattering for odd initial data and an even generic potential by first applying the M\o{}ller wave operator $\Omega$ 
to convert to the flat problem with a pseudodifferential nonlinearity, and then exploiting the special structure of the wave operator for even generic potentials and odd solutions to control commutators between $J_0$ and $\waveop$~\cite{delortModifiedScatteringOdd2016}.  

Several remarkable recent works have shown that the techniques in~\cite{germainNonlinearResonancesPotential2015} can be extended to handle problems involving nonlinear bound states.  Chen has shown in~\cite{chenLongtimeDynamicsSmall2023a} that it is possible to treat the problem with~\eqref{eqn:no-bd-states-assumption} weakened to the assumption that $H$ has only one bound state.  In this case, solutions decompose into a nonlinear bound state and a radiation component that exhibits modified scattering behavior.  Collot and Germain have shown in~\cite{collotAsymptoticStabilitySolitary2023a} that it is possible to extend the methods from~\cite{germainNonlinearResonancesPotential2015} to prove the asymptotic stability of solitons under the assumption that there are no threshold resonances.  This requires working with a matrix version of~\eqref{eqn:main-eqn} and a careful analysis of the modulation parameters for the soliton.  The stability of the soliton for the cubic NLS in dimension 1 under even perturbations was recently proved by Li and L\"uhrmann~\cite{liAsymptoticStabilitySolitary2024a}.

There has also been a great deal of interest in understanding Klein-Gordon equations with potential, since these arise naturally when considering the stability of solitons and kinks~\cite{legerInternalModesRadiation2021,delortLongtimeDispersiveEstimates2022,liSolitonDynamics1D2023,luhrmannAsymptoticStabilitySineGordon2023a,luhrmannCodimensionOneStability2024a,legerInternalModeinducedGrowth2022,kowalczykKinkDynamicsOdd2022,kowalczykKinkDynamicsPhi2016,palaciosLocalEnergyControl2024,cuccagnaSmallEnergyStabilization2023,kowalczykKinkDynamicsOdd2022,kowalczykKinkDynamicsModel2017}.  Compared with the Schr\"odinger equation, the Klein-Gordon equations present additional difficulties, including the existence of internal modes (eigenvectors of $-\Delta + V + 1$ with eigenvalue $\lambda \in (0,1)$)~\cite{delortLongtimeDispersiveEstimates2022,legerInternalModesRadiation2021,legerInternalModeinducedGrowth2022}, and variable coefficient nonlinear terms~\cite{lindbladDecayAsymptoticsOneDimensional2020,lindbladDecayAsymptoticsOneDimensional2020,lindbladAsymptotics1DKleinGordon2021,germainQuadraticKleinGordonEquations2022a,germain1dQuadraticKlein2023,sterbenzDispersiveDecay1D2015,lindbladScatteringKleinGordonEquation2014}.  In particular, it is known that interactions between a zero energy resonance and a variable coefficient quadratic nonlinearity can cause a logarithmic slowdown to the $t^{-1/2}$ linear decay rate of the solution~\cite{lindbladAsymptotics1DKleinGordon2021,lindbladModifiedScattering1D2023}.

\subsection{Main result}

We will prove the following modified scattering result:

\begin{thm}\label{thm:main-theorem}
    Let $V$ satisfy~\cref{eqn:no-bd-states-assumption,eqn:V-decay-assumption}, and suppose that $u(x,t)$ satisfies \eqref{eqn:zero-freq-hypo}.  Then, there exists an $\epsilon_0 > 0$ such that if
    \begin{equation*}
        u(x,1) = e^{iH} u_*(x)
    \end{equation*}
    with
    \begin{equation}\label{eqn:initial-cond-size}
        \lVert \jBra{x} u_* \rVert_{L^2} + \lVert u_* \rVert_{H^1} = \epsilon < \epsilon_0
    \end{equation}
    then equation~\eqref{eqn:main-eqn} has a global solution decaying at the linear rate:
    \begin{equation}\label{eqn:u-decay}
        \lVert u(t) \rVert_{L^\infty} \lesssim t^{-1/2}
    \end{equation}
    More precisely, we have the modified scattering asymptotics
    \begin{equation}\label{eqn:u-asympt}
        u(x,t) = t^{-1/2}\exp\left(-i \frac{x^2}{4t} \mp i \int_1^t \frac{i}{s} |\tilde{u}(-\frac{x}{2t}, s)|^2\;ds\right) u_\infty(x/2t) + o_{t^{1/2}L^\infty_x \cap L^2_x}(1)
    \end{equation}
    where $o_{t^{1/2}L^\infty_x \cap L^2_x}(1)$ denotes a function $R(x,t)$ such that 
    $$\lim_{t \to \infty} t^{1/2} \lVert R \rVert_{L^\infty_x} + \lVert R \rVert_{L^2_x} = 0$$
\end{thm}

Before sketching the argument used to prove~\Cref{thm:main-theorem}, let us make a few remarks on the result and possible generalizations:

\begin{rmk}[Decay rate of the error]
    We can be more explicit about the decay rate of the error term in~\eqref{eqn:u-asympt}: Examining the proof, we find that 
    $$t^{1/2}|R(x,t)| \lesssim \epsilon t^{-1/4+\delta} + \epsilon^3 t^{1-\gamma/2}$$
    and 
    $$\lVert R(x,t) \rVert_{L^2} \lesssim \epsilon t^{-1/2 + \delta} + \epsilon^3 t^{3/4 - \gamma/2}$$
    Here, the terms which are cubic in $\epsilon$ come from approximating the asymptotic profile $\alpha(v,t)$ (see~\eqref{eqn:alpha-eqn}) by $\tilde{u}(-v/2,t)$ in the phase of the modified scattering term.  It is interesting to note that the time decay rate for these phase error terms becomes dominant at exactly the point where the results~\cite{chen1dimensionalNonlinearSchrodinger2022,chen1dCubicNLS2024} derived using the stationary phase method cease to apply.
\end{rmk}

\begin{rmk}[Generic potentials]
    Our argument does not distinguish between generic and exceptional potentials.  Since generic potentials have better low-energy behavior, it is likely that the decay conditions can be lowered further in the generic case.  We plan to investigate this in future work.
\end{rmk}

\begin{rmk}[Assumptions on the nonlinearity]
    Although we will focus on the constant coefficient cubic nonlinearity $|u|^2u$ for simplicity, our methods apply equally well to nonlinearities of the form
    \begin{equation*}
        a(x)|u|^2 u + (W * |u|^2)u
    \end{equation*}
    where $a(x) - 1$ and its derivative decays sufficiently rapidly as $|x| \to \infty$ and $W \in L^1_2$.
\end{rmk}

\subsection{Sketch of the argument}

At a very high level, our argument has two steps:
\begin{enumerate}[(1)]
    \item Assuming that the decay assumption~\eqref{eqn:u-decay} holds, show that $\lVert J_V u \rVert_{L^2} \lesssim \epsilon \jBra{t}^{\delta}$ for some small $\delta > 0$, where $J_V = e^{itH}\DFT^* (i\partial_k) \DFT e^{-itH}$ and $\DFT$ is the distorted Fourier transform (see~\Cref{sec:dist-four-trans}).
    \item Assuming that $\lVert J_V u \rVert_{L^2}$ grows slowly in time, derive an asymptotic equivalent for $u$ that is compatible with~\eqref{eqn:u-decay} and~\eqref{eqn:u-asympt}.
\end{enumerate}
By a standard bootstrap argument, these two ingredients are enough to prove~\Cref{thm:main-theorem}.  We now explain how we obtain the slow growth of the weighted norm and the asymptotic equivalent for $u$.

\subsubsection{Slow growth of the weighted norm}

To obtain slow growth of the weighted norm, we must control the integrals
\begin{equation*}
    \lVert J_V u(t) - e^{i(t-1)H}J_V u(1) \rVert_{L^2} = \left\lVert \int_1^t e^{i(t-s)H} J_V(|u|^2u(s))\;ds \right\rVert_{L^2}
\end{equation*}
Let us briefly review how this is done for the flat NLS
\begin{equation}\label{eqn:flat-NLS}
    i\partial_t u - \Delta u = \pm |u|^2 u
\end{equation}
since our goal is ultimately to emulate this approach as closely as possible.  The weight for the flat Schr\"odinger equation is given by
\begin{equation*}
    J_0 = e^{-it\Delta} \cF^* (i\partial_\xi) \cF e^{it\Delta} = x - 2it\partial_x
\end{equation*}
and an explicit calculation shows that
\begin{equation}\label{eqn:flat-NLS-cubic-ident}
    J_0 (|u|^2 u) = 2|u|^2 J_0 u - u^2 \overline{J_0 u}
\end{equation}
see~\cite{katoNewProofLongrange2011,ifrimGlobalBoundsCubic2015}.  Thus, if we make the bootstrap assumption that $u$ decays like $\epsilon t^{-1/2}$ pointwise, we find that
\begin{equation*}\begin{split}
    \lVert J_0 u(x,t) - e^{-i(t-1)\Delta} J_0 u(x,1) \rVert_{L^2} =& \left\lVert\int_1^t e^{-i(t-s)\Delta} 2|u|^2 J_0 u(s) - u^2 \overline{J_0 u}(s) \;ds\right\rVert_{L^2}\\
    \lesssim& \int_1^t \frac{\epsilon^2}{s} \lVert J_0 u \rVert_{L^2}\;ds
\end{split}\end{equation*}
which shows that $\lVert J_0 u(t) \rVert_{L^2}$ grows no faster than $t^{C\epsilon^2}$.  

As explained earlier, the main difficulty for~\eqref{eqn:main-eqn} is to isolate the resonance structure in the presence of the potential.  To resolve this difficulty, we prove that $J_V$ can be written as
\begin{equation*}
    J_V = T_{V,0} J_0 + E^1_{V,0}(t) + E^2_{V,0}(t) + E^3_{V,0}(t)
\end{equation*}
where $E^j_{V,0}(t)$ are error terms and $T_{V,0}$ is bounded on $L^p$ for $1 < p \leq 2$.  Roughly speaking, $E^1_{V,0}(t)$ has no growth in time and so can be controlled by using the $t^{-1/2}$ pointwise decay of $u$; $\lVert E^2_{V,0}(t)w \rVert_{L^2} \lesssim t\lVert \jBra{x}^{3/2-2\gamma+}w \rVert_{L^2}$, which allows us to use the improved local decay bounds
\begin{equation*}
    \lVert \jBra{x}^{-1} u \rVert_{L^\infty} \lesssim t^{-3/4} \lVert J_v u \rVert_{L^2}
\end{equation*}
to counterbalance the extra growth in time; and $\lVert E^3_{V,0} \rVert_{L^{1+}_x} \lesssim \lVert t \jBra{x}^{-\gamma} w \rVert_{L^{1+}}$, which can be controlled by combining local decay with Strichartz estimate
\begin{equation*}
    \left\lVert \int_1^t e^{i(t-s)H} F(s,x) \;ds \right\rVert_{L^2} \lesssim \lVert F(s,x) \rVert_{L^{4/3+}_s(1,t; L^{1+}_x)}
\end{equation*}
Thus, we find that
\begin{equation*}
    \lVert J_V u(t) - e^{i(t-1)H}J_V u(1) \rVert_{L^2} \leq \left\lVert \int_1^t e^{i(t-s)H} T_{V,0} J_0(|u|^2u(s))\;ds \right\rVert_{L^2} + \{\text{better terms}\}
\end{equation*}
Moreover, we can take advantage of the resonance structure~\eqref{eqn:flat-NLS-cubic-ident} of the flat cubic NLS
to write
\begin{equation*}
    \lVert J_V u(t) - e^{i(t-1)H} J_V u(1) \rVert_{L^2} \lesssim \left\lVert \int_1^t e^{i(t-s)H} T_{V,0} \left(2|u|^2 J_0 u - u^2 \overline{J_0 u}\right)\;ds \right\rVert_{L^2} + \{\text{better}\}
\end{equation*}
To close the estimates, we prove that $J_0$ can be expressed in terms of $J_V$, up to errors:
\begin{equation*}
    J_0 = T_{0,V} J_V + E^1_{0,V}(t) + E^2_{0,V}(t) + E^3_{0,V}(t)
\end{equation*}
The errors $E^j_{0,V}$ have similar behavior to the $E^j_{V,0}$'s ($E^1_{0,V}$ and $E^3_{0,V}$ are in fact dual to $E^1_{V,0}$ and $E^3_{V,0}$, respectively), so the terms $|u|^2 E^j_{0,V} u$ can again be controlled using a combination of $t^{-1/2}$ pointwise decay, improved local decay, and Strichartz estimates.  Thus, the leading-order behavior of $J_Vu$ for~\eqref{eqn:main-eqn} is the same as that for $J_0 u$ in the flat NLS:
\begin{equation}
    \lVert J_V u(t) - e^{i(t-1)H}J_V u(1) \rVert_{L^2} \lesssim \left\lVert \int_1^t e^{i(t-s)H} T_{V,0} \left(2|u|^2 T_{0,V}J_V u - u^2 \overline{T_{0,V}J_V u}\right)\;ds \right\rVert_{L^2} + \{\text{better}\}
\end{equation}
which is enough to prove that $\lVert J_V u \rVert_{L^2}$ can grow only slowly in time.  

\subsubsection{Asymptotic equivalent for \texorpdfstring{$u$}{u}}

With the slow growth of $\lVert J_V u \rVert_{L^2}$ in hand, we must now derive an asymptotic equivalent for $u$.  To do this, we will use the method of testing with wave packets introduced by Ifrim and Tataru in~\cite{ifrimGlobalBoundsCubic2015}.  This method has been used to find modified scattering asymptotics in a number of settings, including water waves~\cite{ifrimTwoDimensionalWater2016,aiTwodimensionalGravityWaves2022,ifrimLifespanSmallData2017a}, the mKdV equation~\cite{harrop-griffithsLongTimeBehavior2016}, the SQG front equation~\cite{aiWellposednessSurfaceQuasigeostrophic2024,aiLowRegularityWellposedness2023}, and the derivative NLS equation~\cite{byarsGlobalDynamicsSmall2024}, among others.  However, to the best of the author's knowledge, it has not previously been used for problems involving potentials.

Intuitively, the idea of testing with wave packets is that the components of the solution of the Schr\"odinger equation localized to energy $\approx |k|^2$ should travel with velocity $\approx -2k$.  Dispersion then leads to a broadening of the wave packet over time, so it has width $\sqrt{t}$ at time $t$.  Since this is a semiclassical approximation, it can only be expected to hold in the classically allowed region where $|k|^2 + V(x) > 0$.  If we make the approximation that $x \approx -2kt$ and that $V(x)$ decays like $|x|^{-2}$ (which is true, at least in an averaged sense, for $V \in L^1_2$), then this amounts to restricting to $|k| \gtrsim t^{-1/2}$, or equivalently $|x| \gtrsim t^{1/2}$, so we cannot expect wave packets to give us meaningful information close to the origin.  Fortunately, for $|x| \lesssim t^{1/2}$, the hypothesis~\eqref{eqn:zero-freq-hypo} actually gives us the improved decay estimate
\begin{equation*}
    |u(x,t)| \lesssim t^{-3/4} \lVert J_V u \rVert_{L^2},\qquad |x| \lesssim t^{1/2}
\end{equation*}
(see~\Cref{lem:improved-low-x-bdds}), which is compatible with~\eqref{eqn:u-decay} and~\eqref{eqn:u-asympt}.  Thus, it is enough to obtain asymptotics in the classically allowed region $|x| \gtrsim t^{1/2}$, where we can expect the wave packets to accurately approximate the solution.

For the flat Schr\"odinger equation, the wave packets are given by
\begin{equation}\label{eqn:wave-packet-def-intro}
    \Psi_v(x,t) = e^{-i\frac{x^2}{4t}} \chi\left(\frac{x - vt}{\sqrt{t}}\right)
\end{equation}
In particular, for $k = -v/2$, $\Psi_v$ behaves approximately like a solution to the free Schr\"odinger equation.  
One might wonder if we need to modify this definition 
to accommodate the potential.  In fact, since we will always work in the region $|x| \gtrsim t^{1/2}$, we can treat the potential perturbatively, allowing us to work with the flat wave packets~\eqref{eqn:wave-packet-def-intro}.  This is related to the fact that the dynamics of $e^{itH}$ approach the free dynamics $e^{-it\Delta}$ in an appropriate sense as $t \to \infty$.  Using the wave packets, we define
\begin{equation}\label{eqn:alpha-def-intro}
    \alpha(v,t) = \int u(x,t) \overline{\Psi_v}(x,t)\;dx
\end{equation}
From~\cite{ifrimGlobalBoundsCubic2015}, we have that $u(x,t) \approx t^{-1/2}e^{i\frac{x^2}{4t}}\alpha(x/t,t)$ provided that $J_0 u$ grows slowly in $L^2_x$, see~\cite[Lemma 2.2]{ifrimGlobalBoundsCubic2015}.  Since $J_V$ and $J_0$ are comparable up to localized errors and we are working in the region $|x| \gtrsim t^{1/2}$, we can effectively replace $J_0$ with $J_V$ (up to lower order errors), which shows that $u(x,t)$ is also well approximated by $t^{-1/2}e^{i\frac{x^2}{4t}}\alpha(x/t,t)$ in our problem.  Thus, the problem of finding asymptotics for $u$ reduces to proving asymptotics for $\alpha$.  By differentiating~\eqref{eqn:alpha-def-intro} and using~\eqref{eqn:main-eqn}, we find that for $|v| \gtrsim t^{-1/2}$
\begin{equation*}
    \partial_t \alpha(v,t) = \mp \frac{i}{t} |\alpha(v,t)| \alpha(v,t) + \{\text{better terms}\}
\end{equation*}
which can be integrated to show that $|\alpha(v,t)| = O(\epsilon)$ and that the long-time asymptotics of $\alpha$ involve a logarithmic phase rotation:
\begin{equation*}
    \alpha(v,t) \approx \exp\left(\int_1^t |\alpha(v,s)|^2 \frac{ds}{s}\right) A(v,t)
\end{equation*}
for some $A$. Recalling the relationship between $\alpha$ and $u$ immediately gives us the decay~\eqref{eqn:u-decay}, and the asymptotics~\eqref{eqn:u-asympt} follow similarly once we note that $|\alpha(v,t)| \approx |\tilde{u}(-v/2,t)|$.

\subsection{Organization of the paper}

The plan for the rest of the paper is as follows: In~\Cref{sec:prelims}, we define some notation, state results on the spectral theory of $H$, and give decay and Strichartz estimates for $e^{itH}$.  Next, we outline the bootstrap argument in~\Cref{sec:bootstrap}.  The weighted estimates we need for the bootstrap are given in~\Cref{sec:weight-ests}, and we derive an asymptotic equivalent for $u$ in~\Cref{sec:asympt}.  Finally, in~\Cref{sec:main-thm-proof} we show how these results combine to prove~\Cref{thm:main-theorem}.

\subsection{Acknowledgements}
The author would like to thank Gong Chen, Tristan L\'eger, and Fabio Pusateri for helpful discussions.

\section{Preliminaries}\label{sec:prelims}

\subsection{Notation and preliminaries}

We write $A \lesssim B$ to mean that $A \leq CB$ for some implicit constant $C$.  If the implicit constant depends on some parameter $P$, we will write $A \lesssim_P B$.  We will also use the Japanese bracket notation
\begin{equation*}
    \jBra{x} := \sqrt{1 + x^2}
\end{equation*}

We denote the (flat) Fourier transform by
\begin{equation*}
    \cF f(\xi) = \hat{f}(\xi) = \frac{1}{\sqrt{2\pi}} \int e^{-ix\xi} f(x)\;dx
\end{equation*}
As is well-known, $\cF$ is unitary on $L^2$, with inverse
\begin{equation*}
    \cF^{-1} g(x) = \cF^* g(x) = \check{g}(x) = \frac{1}{\sqrt{2\pi}} \int e^{ix\xi} g(\xi)\;dx
\end{equation*}
For a function $m: \bbR \to \bbR$, we define the Fourier multiplier $m(D)$ by
\begin{equation*}
    m(D) f := \left[\cF^* m(\xi) \cF f\right] (x)
\end{equation*}
By Plancherel's theorem, $m(D)$ defines a bounded operator on $L^2$ whenever $m \in L^\infty$.  More generally, we have that
\begin{equation*}
    m(D) f (x) =  \sqrt{2\pi} \check{m} * f(x)
\end{equation*}
so if $\check{m} \in L^1$, then $m(D): L^p \to L^p$ for any $p \in [1,\infty]$.

Since we will make frequent use of Hardy's inequality, we record it here for reference:
\begin{lemma}[Hardy's inequality]
    Let $f: \bbR \to \bbR$ be a function with $f(x_0) = 0$.  Then,
    \begin{equation*}
        \left\lVert \frac{f(x)}{x - x_0} \right\rVert_{L^2} \lesssim \lVert \partial_x f \rVert_{L^2} 
    \end{equation*}
\end{lemma}

\subsection{Spectral theory for \texorpdfstring{$-\Delta + V$}{-Laplacian + V}}\label{sec:spectral-theory}

\subsubsection{Jost functions and the scattering matrix}

Here, we recall the basic spectral theory for $-\Delta +V$, and give bounds which will be used throughout the work.  For $V \in L^1_1$, the equation
\begin{equation}\label{eqn:H-eigen-prob}
\left\{\begin{array}{l}
    (-\Delta + V)f_{\pm}(x,k) = k^2 f_{\pm}(x,k)\\
    \lim_{x \to \pm \infty} f_{\pm}(x,k) = e^{\pm ixk}
    \end{array}\right.
\end{equation}
has a solution $f_{\pm}(x,k)$, known as the Jost function.  It is often useful to remove the effect of the oscillations and work with the functions
\begin{equation}\label{eqn:m-pm-def}
    m_{\pm}(x,k) = e^{\mp ixk} f_{\pm}(x,k)
\end{equation}
Then, $m_{\pm}$ solves the Volterra integral equation
\begin{equation}\label{eqn:m-pm-volterra-int}
    m_+(x,k) = 1 \pm \int_{I_\pm(x)} D_k(\pm(y-x))V(y)m_{\pm}(y,k)\;dy
\end{equation}
where $I_+(x) = (x,\infty)$, $I_-(x) = (-\infty,x)$ and
\begin{equation}\label{eqn:Dirichlet-kernel}
    D_k(x) = \int_0^x e^{2ikz}\;dz = \frac{e^{2ikx}-1}{2ik}
\end{equation}
Using the Jost functions, we can define generic and exceptional potentials as follows:
\begin{defn}\label{def:generic-def}
    A potential $V$ is called \textit{generic} if $\int V(x) m_\pm(x,0)\;dx \neq 0$, and \textit{exceptional} otherwise.
\end{defn}
A potential is generic iff $m_\pm(x,0) = f_\pm(x,0)$ is unbounded, which implies that $-\Delta + V$ has no zero energy resonances.

The functions $m_\pm-1$ and their derivatives decay at a rate determined by the spatial decay of the potential.  Defining
\begin{equation*}
    \mathcal{W}_\pm^a(x) = \int_{I_\pm(x)} \jBra{y}^a|V(y)|;dy
\end{equation*} 
we have that
\begin{lemma}\label{lem:m-pm-basic-decay}
    Let $V \in L^1_\gamma$ with $\gamma \geq 1$.  For $0 \leq \theta \leq 1$ and $0 \leq n \leq \gamma-1$, we have that
    \begin{equation}\label{eqn:d-k-m-pm-basic-bdds}\begin{split}
        \left|\partial_k^n \left(m_{\pm}(x,k) - 1\right)\right| \lesssim \frac{1}{|k|^{1-\theta}\jBra{k}^\theta}\begin{cases}
            \mathcal{W}_\pm^{n+\theta}(x) & \pm x \geq -1\\
            \jBra{x}^{n+\theta} & \pm x \leq 1
        \end{cases}
    \end{split}
    \end{equation}
    Moreover, for the $x$ derivative, we have that
    \begin{equation}\label{eqn:dx-d-k-m-pm-bdd}
       |\partial_x \partial_k^n m_\pm(x,k)| \lesssim \begin{cases}
            \mathcal{W}_\pm^n(x) & \pm x \ge -1\\
            \jBra{x}^n & \pm x \leq 1
        \end{cases}
    \end{equation}
\end{lemma}
\begin{proof}
    The proofs for~\eqref{eqn:dx-d-k-m-pm-bdd} and~\eqref{eqn:d-k-m-pm-basic-bdds} with $\theta = 1$ are standard and can be found in~\cite{deiftInverseScatteringLine1979b,wederLpLpEstimates2000,delortModifiedScatteringOdd2016}.  A form  of~\eqref{eqn:d-k-m-pm-basic-bdds} for $\theta \in [0,1)$ appears in~\cite[Lemma 2.1]{wederLpLpEstimates2000} with weaker requirements on the potential but without an explicit decay rate in $x$.  Since this improvement does not appear to be as well known, we give a proof assuming the estimate for $\theta = 1$.  We begin with the case $n = 0$.  Using the Volterra equation~\eqref{eqn:m-pm-volterra-int} and focusing on the bound for $m_+$ for simplicity, we see that
    \begin{equation*}
        m_+(x,k) - 1 = \int_x^\infty D_k(y-x) V(y) m_+(y,k)\;dy
    \end{equation*}
    Now, by~\eqref{eqn:Dirichlet-kernel}, we see that $|D_k(y-x)| \lesssim |k|^{1-\theta}|y-x|^{\theta}$, while~\eqref{eqn:d-k-m-pm-basic-bdds} with $\theta = 1$ implies that $m_+(y,k)$ is bounded uniformly in $y$ and $k$ for $y \geq x \geq -1$.  Thus, for $|k| \leq 1$ it follows that
    \begin{equation*}
        |m_+(x,k) - 1| \lesssim \frac{1}{k^\theta} \int_x^\infty |y|^\theta |V(y)|\;dy \lesssim \frac{\mathcal{W}_+^\theta(x)}{|k|^\theta \jBra{k}^{1-\theta}} 
    \end{equation*}
    while for $|k| > 1$, we use $|D_k(y-x)| \lesssim |k|^{-1}$ to find that
    \begin{equation*}
        |m_+(x,k) - 1| \lesssim \frac{1}{k} \int_x^\infty |V(y)|\;dy \lesssim \frac{\mathcal{W}_+^0(x)}{|k|} \lesssim \frac{\mathcal{W}_+^\theta(x)}{|k|^\theta \jBra{k}^{1-\theta}} 
    \end{equation*}
    For $n\geq 1$, we distribute the derivatives to find that
    \begin{equation*}
        \partial_k^n(m_+(x,k) - 1) = \sum_{n=0}^s \int_x^\infty \partial_{k}^n D_k(y-x) V(y) \partial_k^{s-n} m_+(y,k)\;dy
    \end{equation*}
    and use the bound $|\partial_k^{s-n}m_+(y,k)| \lesssim_V 1$ from the $\theta = 1$ case together with the estimate 
    $$\partial_k^a |D_k(z)| \lesssim \min\left(\frac{\jBra{z}^a}{|k|}, \jBra{z}^{a+1}\right)$$
    to obtain the desired bound.
\end{proof}
In addition to the bounds~\eqref{eqn:dx-d-k-m-pm-bdd}, it will be useful later to have the identity
\begin{equation}\begin{split}\label{eqn:dx-m-ident}
    \partial_x m_+(x,k) =& \int_x^\infty e^{2ik(x-y)} V(y) m_+(y,k)\;dy\\
                        =& \int_x^\infty e^{2ik(x-y)} V(y) (m_+(y,k) - 1)\;dy\\
                        &+ \int_x^\infty e^{2ik(x-y)} V(y) \;dy
\end{split}\end{equation}
which arises naturally in the proof of~\eqref{eqn:dx-d-k-m-pm-bdd}.  The form~\eqref{eqn:dx-m-ident} will prove useful for us later, since each of the individual terms have better properties than suggested by the bound~\eqref{eqn:dx-d-k-m-pm-bdd}:
\begin{enumerate}
    \item The first term contains both $V(y)$ and $m_+(y,k) - 1$, and thus can be shown to decay like $\frac{\jBra{x}^{2\gamma - \theta}}{|k|^{1-\theta} \jBra{k}^\theta}$ using~\eqref{eqn:d-k-m-pm-basic-bdds}.
    \item The second term only depends on $k$ through the complex exponential, which allows us to re-express $x$ integrals in terms of the Fourier transform.
\end{enumerate}
We note that a similar identity holds for $m_-$:
\begin{equation}\label{eqn:dx-m-minus-ident}\begin{split}
    \partial_x m_-(x,k) =& \int_{-\infty}^x e^{2ik(y-x)} V(y) (m_-(y,k) - 1)\;dy\\
                        &+ \int_{-\infty}^x e^{2ik(y-x)} V(y) \;dy
\end{split}\end{equation}

Returning to the eigenvalue problem~\eqref{eqn:H-eigen-prob}, we see that for $k$ nonzero $\{f_\pm(x,k), f_\pm(x,-k)\}$ forms a basis for the solutions of~\eqref{eqn:H-eigen-prob}.  In particular, this implies that we can express $f_\pm(x,k)$ as a linear combination of $f_\mp(x,k)$ and $f_\mp(x,-k)$.  This allows us to define the reflection and transmission coefficients, which for $k \neq 0$ satisfy the equations
\begin{equation}\label{eqn:T-R-def-1}
    f_+(x,k) = \frac{R_-(k)}{T(k)} f_-(x,k) + \frac{1}{T(k)} f_-(x,-k)
\end{equation}
\begin{equation}\label{eqn:T-R-def-2}
    f_-(x,k) = \frac{R_+(k)}{T(k)} f_+(x,k) + \frac{1}{T(k)} f_+(x,-k)
\end{equation}
Intuitively, these equations can be understood as follows: Multiplying through by $T(k)$ and rearranging, we see that 
\begin{equation*}
    f_\pm(x,k) = T(-k) f_\mp(x,k) + R_\pm(-k) f_\pm(x,-k)
\end{equation*}
If we interpret the left-hand side as representing a wave $e^{\pm ikx}$ coming toward the potential from $\pm \infty$ (the incident wave), then the potential reflects a wave $R_\pm(-k) e^{\mp ixk}$ back toward $\pm \infty$ (the reflected wave) and allows a wave $T(-k) e^{\pm ikx}$ to travel to $\mp \infty$ (the transmitted wave).
\begin{rmk}\label{rmk:phys-interp-caveat}
    Strictly speaking, the interpretation above is only correct when $\mp k > 0$, since otherwise the dispersive relation for the equation suggests that the waves will be moving in the opposite directions (that is, the `incident' wave will be moving outward, while the `transmitted' and `reflected' waves will be moving inward).  In particular, it \textit{will not} be valid under our convention for the distorted Fourier transform.
\end{rmk}
Recall that the Wronskian for~\eqref{eqn:H-eigen-prob} is
\begin{equation}\label{eqn:wronskian-def}
    W(f,g) = \partial_x f(x) g(x) - f(x) \partial_x g(x)
\end{equation}
and that it is independent of $x$ for $f$ and $g$ solving~\eqref{eqn:H-eigen-prob}.  By taking Wronskians of various combinations of $f_\pm(x,\pm k)$, we obtain the explicit formulas
\begin{equation}\label{eqn:T-R-exp-formula}\begin{split}
    \frac{1}{T(k)} =& \frac{1}{2ik}W(f_+(x,k), f_-(x,k))
    = 1 - \frac{1}{2ik}\int V(x)m_{\pm}(x,k)\;dx\\
    \frac{R_\pm(k)}{T(k)} =& \frac{1}{2ik}W(f_-(x,\pm k), f_+(x,\mp k))
    = \frac{1}{2ik}\int e^{\mp 2ikx}V(x)m_{\mp}(x,k)\;dx
\end{split}\end{equation}
together with the relations
\begin{equation}\begin{split}
    T(-k) = \overline{T(k)},&\qquad R_\pm(-k) = \overline{R_\pm(k)}\\
    |R_\pm(k)|^2 + |T(k)|^2 = 1, &\qquad T(k)\overline{R_-(k)} + \overline{T(K)} R_+(k) = 0
\end{split}
\end{equation}
In particular, by the third relation, $|T(k)|, |R_\pm(k)| \leq 1$.  Note that by~\Cref{def:generic-def}, $T(0) = 0$ and $R_\pm(0) = -1$ for generic potentials.

A result of~\cite{egorovaDispersionEstimatesOnedimensional2016} implies that $T(D)$ and $R_\pm(D)$ are bounded on Lebesgue spaces:
\begin{lemma}\label{lem:TR-pdo-bd}
    For $V \in L^1_1$, the operators $T(D)$ and $R_\pm(D)$ are bounded on $L^p$ for any $1 \leq p \leq \infty$.
\end{lemma}
We will also need the following derivative bounds:
\begin{lemma}\label{lem:TR-deriv-bdds}
        Let $V \in L^1_2$.  Then,
        \begin{equation}\label{eqn:dTR-bounds}
            |T'(k)|, |R_{\pm}'(k)| \lesssim \jBra{k}^{-1} 
        \end{equation}
\end{lemma}
\begin{proof}
    First, we recall that by~\cite[Theorem 2.2]{egorovaZeroEnergyScattering2015a} $T'$ and $R_\pm'$ belong to the Weiner algebra of functions having an integrable Fourier transform.  It follows that $T',R_\pm'$ are bounded, which gives~\eqref{eqn:dTR-bounds} for $|k| \leq 1$.  To handle the case $k > 1$, we work with the representation formulas~\eqref{eqn:T-R-exp-formula}.  Differentiating the integral representation for $T$, we find that
    \begin{equation*}
        T'(k) = [T(k)]^2 \int_{-\infty}^\infty V(x) \partial_k \left(\frac{m_+(x,k)}{2ik}\right)\;dx
    \end{equation*}
    Now, by~\Cref{lem:m-pm-basic-decay}, we have that $|m_+(x,k)|, |\partial_k m_+(x,k)| \lesssim \frac{\jBra{x}^2}{\jBra{k}}$, so
    \begin{equation*}
        |T'(k)| \lesssim \frac{1}{k^2}\int_{-\infty}^\infty \jBra{x}^2 |V(x)|\;dx
    \end{equation*}
    which is better than required.  Similarly, 
    \begin{equation*}
        R'_\pm(k) = \partial_k \left(T(k) \int_{-\infty}^\infty e^{\mp 2ikx} V(x) m_\mp(x,k)\;dx\right)
    \end{equation*}
    and a similar argument shows that
    \begin{equation*}
        |R'_\pm(k)| \lesssim \frac{1}{k} \int_{-\infty}^\infty \jBra{x}^2 |V(x)|\;dx
    \end{equation*}
\end{proof}

\subsubsection{The distorted Fourier transform}\label{sec:dist-four-trans}

With the Jost solutions in hand, we are now in a position to define the distorted Fourier transform.  Let
\begin{equation}\label{eqn:psi-def}
    \psi(x,k) = \frac{1}{\sqrt{2\pi}}\begin{cases}
        T(k) f_+(x,k) & k \geq 0\\
        T(-k) f_-(x,-k) & k < 0
    \end{cases}
\end{equation}
Then, the distorted Fourier transform of a function $f$ is given by
\begin{equation}\label{eqn:DFT-def}
    \DFT f(k) = \tilde{f}(k) := \int \overline{\psi(x,k)} f(x)\;dx
\end{equation}
Note that for $V = 0$, $\psi(x,k) = (2\pi)^{-1/2} e^{ixk}$ and we recover the flat Fourier transform.  Many of the well-known properties of the flat Fourier transform carry over (with minor adjustments) to the distorted setting.  The following lemma summarizes the properties that will be most useful for us:
\begin{lemma}\label{lem:DFT-basic-prop}
    \begin{enumerate}[(i)]
        \item $\DFT$ is unitary from $L^2_x$ to $L^2_k$, with inverse given by
        \begin{equation*}
            \DFT^{-1} g = \DFT^* g = \int \psi(x,k) g(k)\;dk
        \end{equation*}
        \item If $V$ is generic, then $\tilde{f}(0) = 0$ for $f \in L^1$.
        \item If $V$ is exceptional and $\psi_+(x,0)$ is even (resp. odd), then $\tilde{f}(0) = 0$ for $f$ odd (resp. even).  In particular, $\tilde{f}(0) = 0$ if $V$ is even and $f$ is odd.
        \item For $f \in L^1$, $\tilde{f}(k)$ is continuous for all $k \neq 0$.  In addition, if $\tilde{f}(0) = 0$, then $\tilde{f}(k)$ is continuous at $k = 0$.
        \item $\lVert \partial_k \tilde{f}(k) \rVert_{L^2_k} \lesssim \lVert \jBra{x} f \rVert_{L^2}$
    \end{enumerate}
\end{lemma}
Proofs for each of these facts can be found in~\cite{germainNonlinearSchrodingerEquation2018,agmonSpectralPropertiesSchrodinger1975,dunfordLinearOperators1988,yafaevMathematicalScatteringTheory2010}.  It will often be useful in our proof to write the distorted Fourier transform in terms of incident, reflected, and transmitted waves (see the discussion after~\Cref{eqn:T-R-def-1,eqn:T-R-def-2}).  To this end, we define the \textit{partial distorted Fourier transforms} as
\begin{equation}\label{eqn:partial-DFT-defs}\begin{split}
    \DFT_{+,T} \phi(k) =& \bbOne_{k > 0} \frac{\overline{T(k)}}{\sqrt{2\pi}} \int \chi_+(x) \overline{f_+(x,k)} \phi(x)\;dx\\
    \DFT_{-,T} \phi(k) =& \bbOne_{k < 0} \frac{\overline{T(-k)}}{\sqrt{2\pi}} \int \chi_-(x) \overline{f_-(x,-k)} \phi(x)\;dx\\
    \DFT_{+,I} \phi(k) =& \bbOne_{k > 0} \frac{1}{\sqrt{2\pi}}\int \chi_-(x) \overline{f_-(x,-k)} \phi(x)\;dx\\
    \DFT_{-,I} \phi(k) =& \bbOne_{k < 0}  \frac{1}{\sqrt{2\pi}}\int \chi_+(x) \overline{f_+(x,k)} \phi(x)\;dx\\
    \DFT_{+,R} \phi(k) =& \bbOne_{k > 0} \frac{\overline{R_-(k)}}{\sqrt{2\pi}} \int \chi_-(x) \overline{f_-(x,k)} \phi(x)\;dx\\
    \DFT_{-,R} \phi(k) =& \bbOne_{k < 0} \frac{\overline{R_+(-k)}}{\sqrt{2\pi}} \int \chi_+(x) \overline{f_+(x,-k)} \phi(x)\;dx
\end{split}\end{equation}
where $\chi_+(x) + \chi_-(x) = 1$ and $\chi_{\pm}(x)$ is a smooth function supported in $\pm x > - 1$.  We will sometimes write $\DFT_T = \DFT_{+,T} + \DFT_{-,T}$, and similarly for $\DFT_R,\DFT_I$.  Although the partial DFTs are not unitary, they are bounded on $L^2$:
\begin{lemma}\label{lem:pDFT-bddness}
    For $\gamma > 1$, the partial DFT operators defined in~\eqref{eqn:partial-DFT-defs} are bounded on $L^2$.  Moreover, if $\gamma > 2$, then for $1 < p < 2$ and $A \in \{T,R,I\}$, $\cF^* \DFT_{\pm,A} : L^p \to L^p$.
\end{lemma}
\begin{proof}
    We will give the argument for $\DFT_{+,T}$, since the arguments for the other operators are similar.  Note that we can write
    \begin{equation*}
        \DFT_{+,T} \phi (k) = \bbOne_{k > 0}\overline{T(k)}\hat{\phi}(k) + \bbOne_{k > 0}\overline{T(k)}\int \mathcal{K}_{+,T}(x,k) \phi(x)\;dx
    \end{equation*}
    where
    $$\mathcal{K}_{+,T}(x,k) = \bbOne_{k > 0} \frac{1}{\sqrt{2\pi}} \chi_+(x) (\overline{m_+(x,k)}-1)e^{-ixk}$$
    Since $T(k)$ is bounded, we have that
    \begin{equation*}
        \lVert \bbOne_{k > 0}\overline{T(k)}\hat{\phi} \rVert_{L^2_k} \lesssim \lVert \phi \rVert_{L^2_x}
    \end{equation*}
    For the second term, we note that by~\Cref{lem:m-pm-basic-decay}
    \begin{equation}\label{eqn:calK-bd}
        |\mathcal{K}_{+,T}(x,k)| \lesssim \frac{\jBra{x}^{\theta - \gamma}}{|k|^{1 - \theta} \jBra{k}^{\theta}}
    \end{equation}
    which is $L^2_{x,k}$ for $1/2 < \theta < \gamma - 1/2$.  Thus, by Schur's lemma,
    \begin{equation*}
        \lVert \int \mathcal{K}_{+,T}(x,k) \phi(x)\;dx \rVert_{L^2_k} \lesssim \lVert \phi \rVert_{L^2_x}
    \end{equation*}

    To prove the second part of the theorem, we write
    \begin{equation*}
        \cF^* \DFT_{+,T} \phi = \bbOne_{D > 0} T(D) \phi + \bbOne_{D > 0} T(D) \cF^* \int \overline{(m_+(y,k) - 1)} e^{-iyk} \chi_+(y) \phi(y)\;dy
    \end{equation*}
    Since $\bbOne_{D > 0}$ and $T(D)$ are bounded on $L^p$, it is enough to show that the operator
    \begin{equation*}
        \cF^* \int \overline{(m_+(y,k) - 1)} e^{-iyk} \chi_+(y) \phi(y)\;dy := \cF^* A \phi
    \end{equation*}
    is bounded on $L^p$.  To do this, we will show that $A$ and $\partial_k A$ are bounded from $L^1_x$ to $L^2_k$.  This shows that $A: L^1_x \to \jBra{k}^{-1} L^2_k \subset L^1_k$, and by interpolating with the bound $A: L^2_x \to L^2_k$ proved earlier, we get the result.  The $L^1_x \to L^2_k$ bound for $A$ follows immediately from~\eqref{eqn:calK-bd}, so we focus on the bound for $\partial_k A$.  We have that
    \begin{equation*}\begin{split}
        \partial_k A \phi =& -i\int y \overline{(m_+(y,k) - 1)} e^{-iyk} \chi_+(y) \phi(y)\;dy\\
        &+ \int \overline{\partial_k m_+(y,k)} e^{-iyk} \chi_+(y) \phi(y)\;dy\\
        =& B_1 \phi + B_2 \phi
    \end{split}\end{equation*}
    The first term is simply $y K_{+,T}(y,k)$, and with $\gamma > 2$ the bound~\eqref{eqn:calK-bd} with $\theta = 1$ is sufficient to show that $\lVert y K_{+,T}(y,k) \rVert_{L^2_k L^\infty_y} < \infty$, so $\lVert B_1 \rVert_{L^1_x \to L^2_k} < \infty$.  Turning to $B_2$, we note that
    \begin{equation*}
        |\partial_k m_+(y,k) \chi_+(y)| \lesssim \frac{\jBra{y}^{2-\gamma}}{\jBra{k}} \in L^2_kL^\infty_y
    \end{equation*}
    which gives the desired bound for $B_2$, which completes the proof.
\end{proof}
Note that~\Cref{lem:pDFT-bddness} also implies that the adjoint partial DFTs $\DFT_{\pm,A}^*$ are bounded on $L^2$.  

\subsubsection{Wave operators}

The M\o{}ller wave operator $\waveop$ is defined by
\begin{equation}\label{eqn:wave-op}
    \waveop = \DFT^{*} \mathcal{F}
\end{equation}
Since the flat and distorted Fourier transforms are unitary, so is the wave operator.  In particular, its inverse is given by
\begin{equation}\label{eqn:adj-wave-op}
    \waveop^{-1} = \waveop^* = \cF^{*}\DFT
\end{equation}
For our purposes, it will only be necessary to note that the wave operator and its adjoint are bounded on $L^p$, $1 < p < \infty$:
\begin{lemma}\label{lem:wave-op-bddness}
    For $1 < p < \infty$ and for $\gamma \geq 2$, $\waveop, \waveop^*: L^p \to L^p$.
\end{lemma}
For proof, we refer the reader to~\cite{danconaLpBoundednessWaveOperator2006}.  In general, the endpoints $p = 1$ and $p = \infty$ are excluded, since the wave operator contains a Hilbert transform unless the potential is very exceptional~\cite{wederWkPContinuitySchrodinger1999}.

\subsection{Dispersive estimates}

We now record several dispersive estimates that will be of use to us.  The first is a decay estimate:
\begin{lemma}\label{lem:H-1-infty-decay}
    Suppose $V \in L^1_1$.  For $\phi \in L^1$, we have the estimate
    \begin{equation}\label{eqn:H-1-infty-decay}
        \lVert e^{itH} \phi \rVert_{L^\infty} \lesssim t^{-1/2} \lVert \phi \rVert_{L^1}
    \end{equation}
    Moreover, if instead $\tilde{\phi} \in L^\infty$ and $\partial_k \tilde{\phi} \in L^2$, then
    \begin{equation}\label{eqn:H-J-infty-decay}
        \lVert e^{itH} \phi \rVert_{L^\infty} \lesssim t^{-1/2} \lVert \tilde{\phi} \rVert_{L^\infty} + t^{-3/4} \lVert \partial_k \tilde{\phi} \rVert_{L^2}
    \end{equation}
\end{lemma}
\begin{proof}
    Equation~\eqref{eqn:H-1-infty-decay} follows from~\cite{egorovaDispersionEstimatesOnedimensional2016} (see also~\cite{goldbergDispersiveEstimatesSchrodinger2004}), while~\eqref{eqn:H-J-infty-decay} is proved in~\cite[Proposition 3.1]{germainQuadraticKleinGordonEquations2022a}.
\end{proof}
Since $e^{itH}$ is unitary, the abstract Strichartz estimate of Keel-Tao~\cite{keelEndpointStrichartzEstimates1998} immediately gives us the following lemma:
\begin{lemma}\label{lem:H-strichartz}
    Suppose $V \in L^1_1$.  Then,
    \begin{equation}
        \lVert e^{itH} \phi \rVert_{S} \lesssim \lVert \phi \rVert_{L^2}
    \end{equation}
    and
    \begin{equation}
        \left\lVert \int_0^t e^{i(t-s)H} F(s) \;ds \right\rVert_{L^{\tilde{q}}_tL^{\tilde{r}}_x} \lesssim \lVert F \rVert_{L^{q'}_tL^{r'}_x}
    \end{equation}
    where $(q,r)$ and $(\tilde{q}, \tilde{r})$ are any exponents satisfying the admissibility condition
    \begin{equation}\label{eqn:strich-admiss-cond}
        \frac{1}{r} + \frac{2}{q} = \frac{1}{2},\qquad q,r \geq 2
    \end{equation}
\end{lemma}
If $V \in L^1_2$ and $\tilde{\phi}(0) = 0$, then we get improved decay of $e^{itH} \phi$ near $0$.
\begin{lemma}\label{lem:improved-local-decay}
    Suppose $V \in L^1_2$.  Then, if $\tilde{\phi}(0) = 0$,
    \begin{equation}
        \lVert \jBra{x}^{-1} e^{itH} \phi \rVert_{L^\infty_x} \lesssim t^{-3/4} \left( \lVert \phi \rVert_{L^2} + \lVert \partial_k \phi \rVert_{L^2} \right)
    \end{equation}
\end{lemma}
This result is proved in~\cite{chen1dimensionalNonlinearSchrodinger2022}.  It also follows from~\Cref{lem:H-1-infty-decay} and the following result, which gives local decay in the region $|x| \lesssim t^{1/2}$:
\begin{lemma}\label{lem:improved-low-x-bdds}
    Suppose $V \in L^1_2$, and let $A > 0$. If $\tilde{\phi}(0) = 0$ and $|x| \leq At^{1/2}$, then
    \begin{equation}
        |e^{itH} \phi(x)| \lesssim_A t^{-3/4} \lVert \partial_k \tilde{\phi} \rVert_{L^2}
    \end{equation}
\end{lemma}
To prove~\Cref{lem:improved-low-x-bdds}, we will make use of the following abstract nonstationary phase result (cf.~\cite[Lemma 3.3]{germainNonlinearSchrodingerEquation2018}):
\begin{lemma}\label{lem:abs-nonst-ph}
    Suppose $a(x,k)$ satisfies
    \begin{equation}\label{eqn:a-lem-bd}
        |a(x,k)| + |k||\partial_k a(x,k)| \lesssim 1
    \end{equation}
    uniformly in $x$ and $k$, and suppose that $f \in H^1_k$ is a function vanishing at $k = 0$.  Define
    \begin{equation}\label{eqn:I-abs-st-ph-def}
        \rmI(t,X,x) = \int_0^\infty e^{it(k-X)^2} a(x,k) f(k)\;dk
    \end{equation}
    Then, if $\supp(f) \cap (X - ct^{-1/2}, X + ct^{-1/2}) = \emptyset$,
    \begin{equation}
        |\rmI(t,X,x)| \lesssim_c t^{-3/4} \lVert \partial_k f \rVert_{L^2}
    \end{equation}
    with implicit constant independent of $t,X$ and $x$.
\end{lemma}
\begin{proof}
    We will first prove the case where $X \geq 0$.  Based on the support condition, we can write
    \begin{equation*}\begin{split}
        \rmI(t,X,x) =& \int_{X + ct^{-1/2}}^\infty e^{it(k-X)^2} a(x,k) f(k)\;dk + \int_0^{X - ct^{-1/2}} e^{-it(k-X)^2} a(x,k) f(k)\;dk\\
        =:& \rmI_1 + \rmI_2
    \end{split}\end{equation*}
    with the convention that $\rmI_2 = 0$ if $X < ct^{-1/2}$.  Integration by parts yields that
    \begin{equation*}\begin{split}
        \rmI_1 =& \frac{i}{2t} \int_{X + ct^{-1/2}}^\infty e^{it(k-X)^2} \partial_k \left( \frac{1}{k-X} a(x,k) f(k)\right)\;dk\\
            =& -\frac{i}{2t} \int_{X + ct^{-1/2}}^\infty e^{it(k-X)^2} \frac{1}{(k-X)^2} a(x,k) f(k)\;dk\\
            &+ \frac{i}{2t} \int_{X + ct^{-1/2}}^\infty e^{it(k-X)^2} \frac{1}{k-X} \partial_k a(x,k) f(k)\;dk\\
            &+ \frac{i}{2t} \int_{X + ct^{-1/2}}^\infty e^{it(k-X)^2} \frac{1}{k-X} a(x,k) \partial_k f(k)\;dk\\
            =:& \rmI_{1,1} + \rmI_{1,2} + \rmI_{1,3}
    \end{split}\end{equation*}
    For $\rmI_{1,3}$, Cauchy-Schwarz yields the bound
    \begin{equation*}
        |\rmI_{1,3}| \lesssim \frac{1}{t} \left(\int_{X + ct^{-1/2}}^\infty \frac{dk}{(k-X)^2}\right)^{1/2} \lVert \partial_k f \rVert_{L^2} \lesssim t^{-3/4} \lVert \partial_k f \rVert_{L^2}
    \end{equation*}
    A similar argument holds for $\rmI_{1,1}$, since by Hardy's inequality
    \begin{equation*}
        \left\lVert \frac{f(k)}{X-k}\right\rVert_{L^2} \lesssim \lVert \partial_k f \rVert_{L^2} 
    \end{equation*}
    (recall that $f(X) = 0$).  For $\rmI_{1,2}$, we use the bound~\eqref{eqn:a-lem-bd} to find that
    \begin{equation*}
        |\rmI_{1,2}| \lesssim t^{-1} \int \frac{1}{|k-X|} \frac{|f(k)|}{|k|}\;dk
    \end{equation*}
    which can again be bounded using Hardy's inequality.  A similar argument holds for $\rmI_2$ once we note that $f(X - ct^{-1/2}) = f(0) = 0$, so again there are no boundary terms.

    For $X < 0$, we have
    \begin{equation*}
        I(t,X,x) = \int_{\min(X + ct^{-1/2}, 0)}^\infty e^{it(k-X)^2} a(x,k) f(k)\;dk
    \end{equation*}
    which can be handled using the same argument.
\end{proof}

We now use~\Cref{lem:abs-nonst-ph} to prove~\Cref{lem:improved-low-x-bdds}.
\begin{proof}[Proof of~\Cref{lem:improved-low-x-bdds}]
    To begin, we decompose
    \def\B{2A} 
    \begin{equation*}
        f = \chi\left(\frac{t^{1/2} \sqrt{H}}{2A}\right) f + (1-\chi\left(\frac{t^{1/2} \sqrt{H}}{2A}\right) f) =: f_\text{lo} + f_\text{hi}
    \end{equation*}
    where $\chi(x)$ is a bump function supported on $|x| \leq 2$ with $\chi(x) = 1$ for $x \leq 1$.  Since 
    \begin{equation}\label{eqn:tilde-f-Holder-bd}
        |\tilde{f}(k)| \lesssim |k|^{1/2} \lVert \partial_k \tilde{f} \rVert_{L^2}
    \end{equation}
    whenever $\tilde{f}(0) = 0$, we have that
    \begin{equation*}
        \lVert \tilde{f}_\text{lo} \rVert_{L^\infty} \lesssim t^{-3/4} \lVert \partial_k \tilde{f} \rVert_{L^2} 
    \end{equation*}
    so, by~\Cref{lem:H-1-infty-decay},
    \begin{equation*}
        \lVert e^{itH} f_\text{lo} \rVert_{L^\infty} \lesssim  t^{-3/4} \lVert \partial_k \tilde{f} \rVert_{L^2}
    \end{equation*}
    which is acceptable.

    For $e^{itH} f_\text{hi}$, we will appeal to~\Cref{lem:abs-nonst-ph}.  To see the relevance of this result, we write
    \begin{equation}\label{eqn:f-hi-decomp}\begin{split}
        e^{-itH} f_\text{hi}(x) =& \sum_{\nu \in \{+,-\}} \sum_{S\in \{T,R,I\}} \DFT_{\nu,S}^*(e^{itk^2}\tilde{f}_\text{hi}(k))\\
        =& \sum_{\nu \in \{+,-\}} \sum_{S\in \{T,R,I\}} \int_{\nu k > 0} a_{\nu,S}(x,k) e^{itk^2} e^{\pm ixk} \tilde{f}_\text{hi}(k)\;dk\\
        =& \sum_{\nu \in \{+,-\}}\sum_{S\in \{T,R,I\}} e^{-i\frac{x^2}{4t}} \int_{\nu k > 0} e^{it\left(k\pm\frac{x}{2t}\right)^2} a_{\nu,S}(x,k)  \tilde{f}_\text{hi}(k)\;dk\\
    \end{split}\end{equation}
    where
    \begin{align*}
        a_{+,T}(x,k) =& \frac{1}{\sqrt{2\pi}} \chi_+(x) T(k) m_+(x,k)& a_{-,T}(x,k) =& \frac{1}{\sqrt{2\pi}} \chi_-(x) T(-k)  m_-(x,-k)\\
        a_{+,R}(x,k) =& \frac{1}{\sqrt{2\pi}} \chi_-(x) R_-(k) m_-(x,k)& a_{-,T}(x,k) =& \frac{1}{\sqrt{2\pi}} \chi_+(x) R_+(-k)  m_+(x,-k)\\
        a_{+,I}(x,k) =& \frac{1}{\sqrt{2\pi}} \chi_-(x)  m_-(x,-k)& a_{-,I}(x,k) =& \frac{1}{\sqrt{2\pi}} \chi_+(x)   m_+(x,k)
    \end{align*}
    All of these terms can be seen to satisfy~\eqref{eqn:a-lem-bd} by~\Cref{lem:m-pm-basic-decay,lem:TR-deriv-bdds} and the bounds $|T|, |R_\pm| \leq 1$.  Thus, each of the integral in~\eqref{eqn:f-hi-decomp} can be written in the form~\eqref{eqn:I-abs-st-ph-def} (possibly after reflecting over $k = 0$) with $X = \pm \frac{x}{2t}$.  Since $\tilde{f}_\text{hi}$ vanishes for $|k| \leq \B t^{-1/2}$, we see that the hypotheses of~\Cref{lem:abs-nonst-ph} are satisfied if we take $|x| \leq A t^{1/2}$, and hence
    \begin{equation*}
        |e^{itH} f_\text{hi}(x)| \lesssim t^{-3/4} \lVert \partial_k \tilde{f}_\text{hi} \rVert_{L^2} \lesssim t^{-3/4} \lVert \partial_k \tilde{f} \rVert_{L^2}
    \end{equation*}
    where in the last inequality we have used~\eqref{eqn:tilde-f-Holder-bd} to obtain the bound
    \begin{equation*}\begin{split}
        \lVert \partial_k \tilde{f}_\text{hi} \rVert_{L^2} \lesssim& \left\lVert \left(1-\chi\left(\frac{t^{1/2} \sqrt{H}}{A}\right)\right) \partial_k \tilde{f} \right\rVert_{L^2} + t^{1/2}\left\lVert \chi'\left(\frac{t^{1/2} \sqrt{H}}{A}\right) \tilde{f} \right\rVert_{L^2}\\
        \lesssim& \left(1 + t^{1/2} \left\lVert |k|^{1/2} \chi'\left(\frac{t^{1/2} \sqrt{H}}{A}\right) \right\rVert_{L^2}\right)\lVert \partial_k \tilde{f} \rVert_{L^2}\\
        \lesssim& \lVert \partial_k \tilde{f} \rVert_{L^2}
    \end{split}\end{equation*}
    Combining the low and high energy estimates yields the result.
\end{proof}

Note that the conditions in~\Cref{lem:improved-local-decay} are stated in terms of the initial conditions (or equivalently, in terms of the profile $\phi(t) = e^{itH} u$).  For our purposes, it will be convenient to have conditions that guarantee that $u$ obeys the decay estimates from~\Cref{lem:H-1-infty-decay,lem:improved-local-decay}.  To facilitate this, we define
$$J_V = e^{itH} \widetilde{\mathcal{F}}^{-1} (i\partial_k)  \widetilde{\mathcal{F}} e^{-itH}$$
The previous decay results can be expressed in terms of $J_V$ as follows:
\begin{cor}
    Suppose that $V \in L^1_2$, and that $u, J_V u \in L^2_x$.  Then,
    \begin{equation}\label{eqn:J-glob-decay}
        \lVert u \rVert_{L^\infty} \lesssim t^{-1/2} \left(\lVert u \rVert_{L^2} +  \lVert J_V u \rVert_{L^2}\right)
    \end{equation}
    Moreover, if $\tilde{u}(0,t) = 0$, then
    \begin{equation}\label{eqn:J-loc-decay}
        \lVert \jBra{x}^{-1} u \rVert_{L^\infty} \lesssim t^{-3/4} \left(\lVert u \rVert_{L^2} +  \lVert J_V u \rVert_{L^2}\right)
    \end{equation}
\end{cor}
We will also make use of the flat vector field
$$J_0 = e^{it\Delta} \mathcal{F}^{-1} (i\partial_\xi)  \mathcal{F} e^{-it\Delta}$$
We note that $J_0$ can be written as a local differential operator:
\begin{align*}
    J_0 =& \mathcal{F}^{-1} e^{it\xi^2} (i\partial_\xi) e^{-it\xi^2} \mathcal{F}\\
    =& \mathcal{F}^{-1} (i\partial_\xi + 2t\xi) \mathcal{F}\\
    =& x - 2it\partial_x
\end{align*}

\section{The bootstrap argument}\label{sec:bootstrap}

We now show that~\Cref{thm:main-theorem} follows from a bootstrap argument.  Let us begin by assuming that $u$ satisfies the following bootstrap hypothesis on $[1,T]$:
\begin{equation}\label{eqn:bootstrap-hypos}
    \sup_{t \in [1,T]} \left(t^{1/2} \lVert u(t) \rVert_{L^\infty} + t^{-\delta} \lVert J_V u(t) \rVert_{L^2}\right) \lesssim CM\epsilon
\end{equation}
where $C$ is some constant and $1 \ll M \ll \epsilon^{-2/3}$.  Standard local wellposedness theory guarantees that $u \in C_t H^1_x$, and the conservation of mass and energy gives us that $\sup_t\lVert u(x,t) \rVert_{H^1} \lesssim \epsilon$.  Since $L^\infty \subset H^1$, it follows that the lefthand side of~\eqref{eqn:bootstrap-hypos} is continuous in $T$.  Since
\begin{equation*}
    \lVert u(1) \rVert_{L^\infty} + \lVert J_V u(1) \rVert_{L^2} \lesssim \lVert u_* \rVert_{H^{1,1}} = \epsilon
\end{equation*}
by~\eqref{eqn:initial-cond-size}, it follows that~\eqref{eqn:bootstrap-hypos} holds for some $T > 1$.  

Assuming that~\eqref{eqn:bootstrap-hypos} holds, we will prove that
\begin{equation}\label{eqn:bootstrap-up}
    \sup_{t \in [1,T]} \left(t^{1/2} \lVert u(x,t) \rVert_{L^\infty(|x| \geq 100 t^{1/2})} + t^{-\delta} \lVert J_V u(t) \rVert_{L^2} \right) \lesssim C\epsilon
\end{equation}
where $0 < \delta < \min\left(1/4,\frac{\gamma - 2}{3}\right)$.  We first observe that~\eqref{eqn:bootstrap-up} (which, importantly, does not have a factor of $M$ on the right-hand side) implies that~\eqref{eqn:bootstrap-hypos} holds for some $T' > T$.  Indeed, we immediately have control on $\lVert J_V u(t) \rVert_{L^2}$ with an improved constant, and by~\Cref{lem:improved-low-x-bdds} this translates to a decay rate of $C\epsilon t^{-3/4+\delta} \ll C\epsilon t^{-1/2}$ in the region $|x| \leq 100 t^{1/2}$, which gives~\eqref{eqn:u-decay}.  By a standard continuity argument, we conclude that~\eqref{eqn:bootstrap-up} holds with $T = \infty$.  Moreover, the method we use to prove decay for $|x| \geq 100 t^{1/2}$ also yields~\eqref{eqn:u-asympt}, so~\Cref{thm:main-theorem} follows.

The remainder of the paper will focus on proving the bootstrap estimates~\eqref{eqn:bootstrap-up}.  We prove the weighted estimate for $J_V u$ in~\Cref{sec:weight-ests} and the pointwise estimate in~\Cref{sec:asympt}.
 
\section{Weighted estimates}\label{sec:weight-ests}

\subsection{Statement of results}

Here, we will prove that $J_V u$ can grow only slowly in time:
\begin{prop}\label{prop:weighted-est}
    Suppose $\epsilon > 0$ is sufficiently small, and that $u$ satisfies the bootstrap hypotheses~\eqref{eqn:bootstrap-hypos} on $[1,T]$.  Then,
    \begin{equation}\label{eqn:JV-weighted-bdd}
        \sup_{t \in [1,T]} \lVert J_V u \rVert_{L^2} \leq C \epsilon t^{\delta}
    \end{equation}
\end{prop}

The main difficulty in proving weighted estimates for~\eqref{eqn:main-eqn} is that the potential effectively `smears out' the resonances from the free problem by allowing a much greater range of multilinear interactions in distorted frequency space.  These new interactions are ultimately lower order, but we can only take advantage of this fact by carefully separating them from the leading-order interactions.  
Our approach consists in showing that we can approximate $J_V$ by $J_0$ \textit{as operators} (and vice versa), up to errors that are lower order in our nonlinear application.
\begin{thm}\label{thm:J0-to-JV}
    Suppose that $V \in L^1_\gamma$ with $\gamma > 2$.  Then, there exist operators $T_{0,V}$, and $E^j_{0,V}(t)$, $j = 1,2,3$, such that
    \begin{equation}\label{eqn:J0-JV-decomp}
        J_0 = T_{0,V} J_V + E^1_{0,V}(t) + E^2_{0,V}(t) + E^3_{0,V}(t)
    \end{equation}
    where
    \begin{equation}\label{eqn:T-0V-expr}
        T_{0,V} = \DFT_{T}^{*}\DFT + \DFT_{I}^{*}\DFT -\DFT_{R}^{*}\DFT 
    \end{equation}
    is time independent and bounded on $L^2$.  Moreover, if $\tilde{w}(0) = 0$, then the errors $E^j_{0,V}$ satisfy
    \begin{equation}\label{eqn:J0-to-JV-errs}\begin{split}
        \lVert E^1_{0,V}(t) w \rVert_{L^2 \to L^2} \lesssim& \lVert w \rVert_{L^2}\\
        \lVert \jBra{x}^{3/2 - 2\gamma - a} E^2_{0,V}(t) w\rVert_{L^2} \lesssim& t^{1/4}\lVert J_V w \rVert_{L^2}\\
        \lVert \jBra{x}^{\gamma} E^3_{0,V}(t) w \rVert_{L^p} \lesssim& t \lVert w \rVert_{L^p}
    \end{split}\end{equation}
    for any $p \in (1,\infty)$ and $a > 0$ provided that $\tilde{w}(0) = 0$.

\end{thm}

\begin{thm}\label{thm:JV-to-J0}
    Suppose that $V \in L^1_\gamma$ with $\gamma > 2$.  Then, there exist operators $T_{V,0}$, and $E^j_{0,V}(t)$, $j = 1,2,3$, such that
    \begin{equation}\label{eqn:JV-J0-decomp}
        J_V = T_{V,0} J_0 + E^1_{V,0}(t) + E^2_{V,0}(t) + E^3_{V,0}(t)
    \end{equation}
    where 
    \begin{equation}\label{eqn:T-V0-expr}
        T_{V,0} = \DFT^{-1}(\DFT_{T} + \DFT_I - \DFT_R)
    \end{equation} is bounded on $L^p$ for $p \in (1,2]$ and the errors $E^j_{0,V}$ satisfy
    \begin{equation}\label{eqn:JV-to-J0-errs}\begin{split}
        \lVert E^1_{V,0}(t) w \rVert_{L^2 \to L^2} \lesssim& \lVert w \rVert_{L^2}\\
        \lVert E^2_{V,0}(t) w\rVert_{L^2} \lesssim& t\lVert \jBra{x}^{3/2 + a - 2\gamma} w \rVert_{L^2}\\
        \lVert  E^3_{V,0}(t) w \rVert_{L^p} \lesssim& t \lVert \jBra{x}^{-\gamma} w \rVert_{L^p}
    \end{split}\end{equation}
    for any $p \in (1,\infty)$ and $a > 0$.
\end{thm}



\subsection{Slow growth of the weighted norm}

Assuming for now that these two results hold, let us show that they imply slow growth of $J_V u$ given in~\Cref{prop:weighted-est}.
\begin{proof}[Proof of~\Cref{prop:weighted-est}]
    First, we note that the hypotheses of~\Cref{thm:J0-to-JV,thm:JV-to-J0} are satisfied by~\eqref{eqn:V-decay-assumption} and that $\tilde{w}(0) = 0$ by~\eqref{eqn:zero-freq-hypo}.  Using~\eqref{eqn:main-eqn} and recalling that $J_V$ commutes with $H$, we see that
    \begin{equation*}
        J_V u(t) =  e^{itH}J_V u_*  - i \int_1^t e^{i(t-s)H} J_V (|u|^2 u)(s)\;ds
    \end{equation*}
    Applying~\Cref{thm:J0-to-JV,thm:JV-to-J0} and recalling that $J_0$ satisfies the Leibniz-like identity
    \begin{equation*}
        J_0 (|u|^2u) = 2 |u|^2 J_0 u - u^2 \overline{J_0 u}
    \end{equation*}
    we can expand the Duhamel term to find that
    \begin{subequations}\begin{align}
        \int_1^t e^{i(t-s)H} J_V (|u|^2 u)(s)\;ds =& \int_1^t e^{i(t-s)H} T_{V,0} (2|u|^2 T_{0,V} J_V u(s) - u^2 \overline{T_{0,V} J_V u}(s) \;ds \label{eqn:duhamel-exp-main}\\
        &+ \int_1^t e^{i(t-s)H} T_{V,0} |u|^2 E^1_{0,V}(s) u(s)\;ds \label{eqn:duhamel-exp-inner-E-1}\\
        &+ \int_1^t e^{i(t-s)H} T_{V,0} |u|^2 E^2_{0,V}(s) u(s)\;ds \label{eqn:duhamel-exp-inner-E-2}\\
        &+ \int_1^t e^{i(t-s)H} T_{V,0} |u|^2 E^3_{0,V}(s) u(s)\;ds \label{eqn:duhamel-exp-inner-E-3}\\
        &+ \int_1^t e^{i(t-s)H} E^1_{V,0} |u|^2 u(s)\;ds \label{eqn:duhamel-exp-outer-E-1}\\
        &+ \int_1^t e^{i(t-s)H} E^2_{V,0} |u|^2 u(s)\;ds \label{eqn:duhamel-exp-outer-E-2}\\
        &+ \int_1^t e^{i(t-s)H} E^3_{V,0} |u|^2 u(s)\;ds \label{eqn:duhamel-exp-outer-E-3}
    \end{align}
    \end{subequations}
    Using the $L^2$ boundedness of $T_{V,0}$ and $T_{0,V}$, we immediately see that the leading term~\eqref{eqn:duhamel-exp-main} can be estimated as
    \begin{equation*}\begin{split}
        \lVert \eqref{eqn:duhamel-exp-main} \rVert_{L^2} \lesssim& \int_1^t \lVert u (s) \rVert_{L^\infty}^2 \lVert J_V u (s) \rVert_{L^2}\;ds
        \lesssim M^3\epsilon^3 \int_1^t s^{-1+\delta}\;ds
        \lesssim M^3\epsilon^3 t^{\delta}
    \end{split}\end{equation*}
    which is consistent with the bootstrap estimate~\eqref{eqn:JV-weighted-bdd}.  A similar argument using the bounds for $E^1_{0,V}$ and $E^1_{V,0}$ shows that~\eqref{eqn:duhamel-exp-inner-E-1} and~\eqref{eqn:duhamel-exp-outer-E-1} satisfy the same bounds.
    For~\eqref{eqn:duhamel-exp-outer-E-2}, we use the improved local decay from~\Cref{lem:improved-local-decay} to obtain the bound
    \begin{equation*}\begin{split}
        \lVert \eqref{eqn:duhamel-exp-outer-E-2} \rVert_{L^2} \lesssim& \int_1^t s\lVert \jBra{x}^{3/2 + a - 2\gamma} |u|^2u (s) \rVert_{L^2}\;ds\\
        \lesssim& \int_1^t s\lVert \jBra{x}^{-2-\mu} |u|^2u (s) \rVert_{L^\infty}\;ds\\
        \lesssim& \int_1^t s \lVert \jBra{x}^{-1} u \rVert_{L^\infty}^2 \lVert \jBra{x}^{-\mu} u \rVert_{L^\infty}\;ds\\
        \lesssim& \int_1^t s^{-1-\mu/2} (\lVert J_V u \rVert_{L^2} + \lVert u \rVert_{L^2})^3\;ds\\
        \lesssim& M^3\epsilon^3
    \end{split}\end{equation*}
    where we have chosen $\mu \in (6\delta, 2\gamma - 4 - a)$ (note that this range is nonempty for $a$ sufficiently small).  Similarly, we have that
    \begin{equation*}\begin{split}
        \lVert \eqref{eqn:duhamel-exp-inner-E-2} \rVert_{L^2} \lesssim& \int_1^t s^{1/4}\lVert \jBra{x}^{3/2 + a - 2\gamma} |u|^2 \rVert_{L^\infty} \lVert J_V u (s) \rVert_{L^2}\;ds\\
        \lesssim& \int_1^t s^{1/4} \lVert \jBra{x}^{-1} u \rVert_{L^\infty}^2 \lVert J_V u \rVert_{L^2}\;ds\\
        \lesssim& M^3\epsilon^3
    \end{split}\end{equation*}
    For~\eqref{eqn:duhamel-exp-outer-E-3}, we use the Strichartz estimates given in~\Cref{lem:H-strichartz} to find that
    \begin{equation*}\begin{split}
        \lVert \eqref{eqn:duhamel-exp-outer-E-3} \rVert_{L^2} \lesssim& \left\lVert s \lVert \jBra{x}^{-\gamma} |u|^2u \rVert_{L^{r'}_x} \right\rVert_{L^{q'}(1,t;)}\\
        \lesssim& \left\lVert s \lVert \jBra{x}^{1/r' + a} \rVert_{L^{r'}_x} \lVert \jBra{x}^{\gamma - 1/r' - a} |u|^2 u \rVert_{L^\infty}\right\rVert_{L^{q'}(0,t)}\\
        \lesssim& M^3\epsilon^3 \lVert s^{-1/2 + (\delta - 1/4) (\gamma -1/r' - a)}\rVert_{L^{q'}(1,t)}
    \end{split}\end{equation*}
    where $a > 0$ is arbitrary and $(q,r)$ satisfies~\eqref{eqn:strich-admiss-cond}.  For $\delta < 1/4$, it is always possible to find and admissible $(q,r)$ such that
    \begin{equation}\label{eqn:exponent-cond}
        \frac{1}{q'} < \frac{1}{2} + \left(\frac{1}{4} - \delta\right) \left(\gamma - \frac{1}{r'} - a\right)
    \end{equation}
    so
    \begin{equation*}
        \lVert \eqref{eqn:duhamel-exp-outer-E-3} \rVert_{L^2} \lesssim M^3 \epsilon^3
    \end{equation*}
    A similar argument shows that
    \begin{equation*}\begin{split}
        \lVert \eqref{eqn:duhamel-exp-inner-E-3} \rVert_{L^2} \lesssim& \left\lVert s \lVert |u|^2  E_{0,V}^3 u (s) \rVert_{L^{r'}_x}\right\rVert_{L^{q'}(1,t)}\\
        \lesssim& \left\lVert s \lVert \jBra{x}^{-\gamma} |u|^2 \rVert_{L^{\rho}_x} \lVert \jBra{x}^{\gamma} E_{0,V}^3 u \rVert_{L^{R}_x} \right\rVert_{L^{q'}(1,t)}\\
        \lesssim& M^3\epsilon^3\left\lVert s^{-\frac{1}{2} + \frac{1}{R} - \left(\frac{1}{4} - \delta\right) \left( \gamma - \frac{1}{\rho} - a \right)} \right\rVert_{L^{q'}(1,t)}  
    \end{split}\end{equation*}
    where $\frac{1}{R} + \frac{1}{\rho} = \frac{1}{r'}$.  Choosing $R$ sufficiently large and using~\eqref{eqn:exponent-cond}, we find that $\lVert \eqref{eqn:duhamel-exp-inner-E-3} \rVert_{L^2} \lesssim M^3\epsilon^3$, which completes the proof of~\Cref{prop:weighted-est}.
\end{proof}

Thus, it only remains to prove~\Cref{thm:J0-to-JV,thm:JV-to-J0}.  We will prove~\Cref{thm:J0-to-JV} in~\Cref{sec:J0-to-JV} and~\Cref{thm:JV-to-J0} in~\Cref{sec:JV-to-J0}.

\subsection{Control of \texorpdfstring{$J_0$}{J0} in terms of \texorpdfstring{$J_V$}{JV}}\label{sec:J0-to-JV}
\begin{proof}[Proof of~\Cref{thm:J0-to-JV}]
    If $w = e^{itH} \phi$, then based on the 
    inversion formula from~\Cref{lem:DFT-basic-prop}, we have
    \begin{align*}
        J_0 w =& (x - 2it\partial_x) \DFT^*[e^{itk^2}\tilde{\phi}(k)](x)\\
              =& (x - 2it\partial_x) (\DFT^*_{+,T} + \DFT^*_{-,T} + \DFT^{*}_{+,R} + \DFT^{*}_{-,R} + \DFT^{*}_{+,I} + \DFT^{*}_{-,I})[e^{-itk^2}\tilde{\phi}(k)](x)
    \end{align*}
    Thus, it suffices to examine the action of $J_0$ on the transmitted, reflected and incident partial Fourier transforms.  For simplicity, we will only consider the partial DFT's over positive wavenumbers: The vanishing of the DFT at $k = 0$ ensures that we do not pick up boundary terms when integrating by parts, and the contribution from $k < 0$ is similar.  We will also focus primarily on the transmitted wave and then explain the modifications necessary to handle the reflected and incident waves.  We compute that
    \begin{align*}
        J_0 \DFT^{*}_{+,T} \tilde{w} 
        =& \frac{1}{\sqrt{2\pi}} \int_{\bbR^+} T(k) \chi_+(x) m_+(x,k) \left[(x - 2it\partial_x)e^{i(kx + tk^2)}\right] \tilde{\phi}(k)\;dk\\
        &- \frac{2it}{\sqrt{2\pi}} \int_{\bbR^+} T(k) \partial_x(\chi_+(x) m_+(x,k)) e^{i(kx + tk^2)} \tilde{\phi}(k)\;dk
    \end{align*}
    Focusing on the first term for a moment, we observe that
    \begin{equation}\label{eqn:fwd-wave-IBP}
        (x - 2it\partial_x)e^{i(kx + tk^2)}  = (x + 2tk) e^{i(kx + tk^2)}  = -i\partial_k e^{i(kx + tk^2)}
    \end{equation}
    Thus, we can integrate by parts to find that
    \begin{equation}\label{eqn:J0-JV-T-split}\begin{split}
        J_0 \DFT^{*}_{+,T} w =& \frac{1}{\sqrt{2\pi}} \int_{\bbR^+} \chi_+(x) T(k) m_+(x,k) e^{i(kx + tk^2)} (i\partial_k \tilde{\phi}(k))\;dk\\ 
        & + \frac{i}{\sqrt{2\pi}} \int_{\bbR^+} \partial_k (T(k) \chi_+(x) m_+(x,k)) e^{ikx} \tilde{w}(k)\;dk\\
        &- \frac{2it}{\sqrt{2\pi}} \int_{\bbR^+} T(k) \partial_x (\chi_+(x) m_+(x,k)) e^{ikx} \tilde{w}(k)\;dk\\
        =:& \DFT_{+,T}^{+}\DFT [J_V w] + i\rmI_{0,V,T} - 2i \rmII_{0,V,T}
    \end{split}\end{equation}
    $\DFT_{+,T}^{-1}\DFT J_V$ is one of the terms making up $T_{0,V}J_V$, so it only remains to show that the terms $\rmI_{0,V,T}$ and $\rmII_{0,V,T}$ satisfy bounds compatible with~\eqref{eqn:J0-to-JV-errs}.  For $\rmI_{0,V,T}$, we write
    \begin{equation}\label{eqn:rmI-T-decomp}\begin{split}
        \rmI_{0,V,T} =& \frac{\chi_+(x)}{\sqrt{2\pi}} \int_{\bbR^+} T'(k) e^{ikx} \tilde{w}(k)\;dk\\
        &+ \frac{1}{\sqrt{2\pi}} \int_{\bbR^+} T'(k)  \chi_+(x) (m_+(x,k)-1) e^{ikx} \tilde{w}(k)\;dk\\
        &+ \frac{1}{\sqrt{2\pi}} \int_{\bbR^+} T(k) \chi_+(x) \partial_km_+(x,k) e^{ikx} \tilde{w}(k)\;dk\\
        =:& \rmI_{0,V,T}^{(1)} + \rmI_{0,V,T}^{(2)} + \rmI_{0,V,T}^{(3)}
    \end{split}
    \end{equation}
    For the first term, we note that $T'(k)$ is bounded by~\Cref{lem:TR-deriv-bdds}, so by Plancherel's theorem
    \begin{equation*}
        \lVert \rmI_{0,V,T}^{(1)}\rVert_{L^2_x} \lesssim \lVert T'(k) \tilde{w}(k) \rVert_{L^2_k} \lesssim \lVert w \rVert_{L^2}
    \end{equation*}
    For $\rmI^{(2)}_{0,V,T}$, we note that by~\Cref{lem:TR-deriv-bdds,lem:m-pm-basic-decay},
    $$|T'(k) \chi_+(x) (m_+(x,k) - 1)| \lesssim \chi_+(x)\frac{\mathcal{W}^{1}_+(x)}{\jBra{k}^2} \in L^2_{x,k}$$
    which gives the bound $|\rmI^{(2)}_{0,V,T}| \lesssim \lVert w \rVert_{L^2}$.  Finally, for $\rmI^{(3)}_{0,V,T}$, we observe that
    $$|T(k) \chi_+(x) \partial_k m_+(x,k)|\lesssim \chi_+(x) \frac{\mathcal{W}^{3/2+}_+(x)}{|k|^{1/2-} \jBra{k}^{1/2+}} \in L^2_{x,k}$$
    so
    \begin{equation*}
        \lVert \rmI^{(3)}_{0,V,T} \rVert_{L^2_x} \lesssim \lVert u \Vert_{L^2}
    \end{equation*}
    In particular, the bounds for $\rmI_{0,V,T}$ are compatible with those for $E^1_{0,V}$.  To control $\rmII_{0,V,T}$, we take advantage of dispersion to compensate for the extra factor of $t$.  To begin, we write
    \begin{equation*}\begin{split}
            \rmII_{0,V,T} =& \frac{t}{\sqrt{2\pi}} \int_{\bbR^+} T(k) (\partial_x \chi_+(x))m_+(x,k)) e^{ikx} \tilde{w}(k)\;dk\\
            &+ \frac{t}{\sqrt{2\pi}} \int_{\bbR^+} T(k) \chi_+(x) (\partial_x m_+(x,k)) e^{ikx} \tilde{w}(k)\;dk\\
            =:& \rmII_{0,V,T}^{(1)} + \rmII_{0,V,T}^{(2)}
    \end{split}
    \end{equation*}
    We will show how to estimate $\rmII^{(2)}_{0,V,T}$: The estimate for $\rmII^{(1)}_{0,V,T}$ is simpler, since the compact support of $\partial_x \chi_+$ allows us to apply arbitrarily many weights.  Using~\eqref{eqn:dx-m-ident}, we further decompose $\rmII^{(2)}_{0,V,T}$ as
    \begin{equation*}\begin{split}
        \rmII_{0,V,T}^{(2)} =& \frac{t}{\sqrt{2\pi}} \int_{\bbR^+} \int_x^\infty \chi(t^{1/2}k) T(k) \chi_+(x) V(y) (m_+(y,k) - 1) e^{ik(x-2y)} \tilde{w}(k)\;dydk\\
        &+ \frac{t}{\sqrt{2\pi}} \int_{\bbR^+} \int_x^\infty (1-\chi(t^{1/2}k)) T(k) \chi_+(x) V(y) (m_+(y,k) - 1) e^{ik(x-2y)} \tilde{w}(k)\;dydk\\
        &+ \frac{t}{\sqrt{2\pi}} \int_{\bbR^+} \int_x^\infty T(k) \chi_+(x) V(y) e^{ik(x-2y)} \tilde{w}(k)\;dydk\\
        =& \rmII^{(2,1)}_{0,V,T} + \rmII^{(2,2)}_{0,V,T} + \rmII^{(2,3)}_{0,V,T}
    \end{split}\end{equation*}
    where $\chi(z)$ is supported on $|z| < 2$.  For $\rmII^{(2,1)}_{0,V,T}$, we observe that
    \begin{equation}\label{eqn:double-V-trick}\begin{split}
        \chi_+(x)  \int_x^\infty |V(y) (m_+(y,k) - 1)|\;dy \lesssim& \frac{\jBra{x}^{1-2\gamma}}{\jBra{k}}
    \end{split}
    \end{equation}
    Using the bound $|\tilde{w}(k)| \lesssim |k|^{1/2} \lVert J_V w \rVert_{L^2}$ and the support condition for $\chi$, it follows that
    \begin{equation*}
        \lVert \jBra{x}^{2\gamma - 3/2 -a} \rmII^{(2,1)}_{0,V,T} \rVert_{L^2} \lesssim t^{1/4}\lVert J_V w \rVert_{L^2}
    \end{equation*}
    which is compatible with the estimate required for $E^2_{0,V}(t)$.  For $\rmII^{(2,2)}_{0,V,T}$, we write $\tilde{w}(k) = e^{itk^2} \tilde{\phi}(k)$ and integrate by parts using $\frac{1}{2itk} \partial_k e^{itk^2} = e^{itk^2}$ to obtain
    \begin{subequations}\begin{align}
        \rmII^{(2,2)}_{0,V,T} =& -\frac{t^{\frac{1}{2}}}{2i\sqrt{2\pi}} \int_{\bbR^+} \int_x^\infty \chi'(t^{\frac{1}{2}}k) T(k) \chi_+(x) V(y) (m_+(y,k) - 1) e^{ik(x-2y)} e^{itk^2}\frac{\tilde{\phi}(k,t)}{k}\;dydk\label{eqn:rmII-22-T-bdy}\\
        &+ \frac{1}{2i\sqrt{2\pi}} \int_{\bbR^+} \int_x^\infty (1-\chi(t^{\frac{1}{2}}k)) T'(k) \chi_+(x) V(y) (m_+(y,k) - 1) e^{ik(x-2y)} e^{itk^2}\frac{\tilde{\phi}(k,t)}{k}\;dydk\label{eqn:rmII-22-T-T}\\
        &+ \frac{1}{2i\sqrt{2\pi}} \int_{\bbR^+} \int_x^\infty (1-\chi(t^{\frac{1}{2}}k)) T(k) \chi_+(x) V(y) \partial_k\left[(m_+(y,k) - 1) e^{ik(x-2y)}\right] e^{itk^2}\frac{\tilde{\phi}(k,t)}{k}\;dydk\label{eqn:rmII-22-T-Jost}\\
        &+ \frac{1}{2i\sqrt{2\pi}} \int_{\bbR^+} \int_x^\infty (1-\chi(t^{\frac{1}{2}}k)) T(k) \chi_+(x) V(y) (m_+(y,k) - 1) e^{ik(x-2y)} e^{itk^2}\frac{\partial_k\tilde{\phi}(k,t)}{k}\;dydk\label{eqn:rmII-22-T-phi}\\
        &- \frac{1}{2i\sqrt{2\pi}} \int_{\bbR^+} \int_x^\infty (1-\chi(t^{\frac{1}{2}}k)) T(k) \chi_+(x) V(y) (m_+(y,k) - 1) e^{ik(x-2y)} e^{itk^2}\frac{\tilde{\phi}(k,t)}{k^2}\;dydk\label{eqn:rmII-22-T-denom}
    \end{align}
    \end{subequations}
    For~\eqref{eqn:rmII-22-T-bdy}, we note that $\lVert \frac{\tilde{\phi(k)}}{k}\rVert_{L^2} \lesssim \lVert J_V w \rVert_{L^2}$ by Hardy's inequality, so using the bound~\eqref{eqn:double-V-trick} we have that
    \begin{equation*}\begin{split}
        \lVert \jBra{x}^{2\gamma - 3/2 - a} \eqref{eqn:rmII-22-T-bdy} \rVert_{L^2}
        \lesssim& t^{1/2} \int_{k > 0}  \frac{|\chi'(t^{1/2}k) T(k)|}{\jBra{k}} \lVert \jBra{x}^{-1/2-a} \rVert_{L^2_x} \left|\frac{\tilde{\phi}(k,t)}{k}\right|\; dk\\
        \lesssim& t^{1/4} \lVert J_V w \rVert_{L^2}
    \end{split}\end{equation*}
    For~\eqref{eqn:rmII-22-T-T}, we note that~\Cref{lem:TR-deriv-bdds} implies that $T'(k) \in L^2$ for $\gamma \geq 2$, so
    \begin{equation*}
        \lVert \jBra{x}^{2\gamma - 3/2 - a} \eqref{eqn:rmII-22-T-T} \rVert_{L^2} \lesssim \lVert J_V w \rVert_{L^2}
    \end{equation*}
    For~\eqref{eqn:rmII-22-T-Jost}, we use~\eqref{eqn:d-k-m-pm-basic-bdds} with $\theta = 0$ to conclude that
    \begin{equation*}\begin{split}
        \chi_+(x)  \int_x^\infty |V(y) \partial_k((m_+(y,k) - 1)e^{ik(x-2y)})|\;dy \lesssim& \frac{\jBra{x}^{2\gamma - 1}}{|k|}
    \end{split}
    \end{equation*}
    so
    \begin{equation*}\begin{split}
        \lVert \jBra{x}^{2\gamma - 3/2 - a} \eqref{eqn:rmII-22-T-Jost} \rVert_{L^2} 
        \lesssim& \int_{k > 0} (1-\chi(t^{1/2}k)) \frac{1}{|k|} \lVert \jBra{x}^{-1/2-a} \rVert_{L^2_x} \left|\frac{\tilde{\phi}(k,t)}{k}\right|\; dk\\
        \lesssim& t^{1/4} \lVert J_V w \rVert_{L^2}
    \end{split}\end{equation*}
    and similar arguments show that
    \begin{equation*}
        \lVert \jBra{x}^{2\gamma -3/2-a} (\eqref{eqn:rmII-22-T-phi} + \eqref{eqn:rmII-22-T-denom}) \rVert_{L^2} \lesssim t^{1/4} \lVert J_V w \rVert_{L^2}
    \end{equation*}
    which shows that $\rmII^{(2,2)}_{0,V,T}$ can be absorbed into $E^2_{0,V}$.  
    
    Finally, for $\rmII^{(2,3)}_{0,V,T}$, we interchange the order of integration and identify the $k$ integral as a Fourier integral operator to obtain
    \begin{equation*}
        \rmII^{(2,3)}_{0,V,T} = t \chi_+(x) \int_x^\infty V(y) \bbOne_{D>0} T(D) \waveop^* w(x-2y)\;dy
    \end{equation*}
    where $\waveop^*$ is the adjoint wave operator defined in~\eqref{eqn:adj-wave-op}.  Thus, for $p \in (1,\infty)$, we find that
    \begin{equation*}\begin{split}
        \lVert \jBra{x}^\gamma \rmII^{(2,3)}_{0,V,T} \rVert_{L^p} \leq& t \int \jBra{y}^\gamma |V(y)|\;dy \lVert \waveop^* u \rVert_{L^p}\\
        \lesssim& t\lVert u \rVert_{L^p}
    \end{split}\end{equation*}
    where on the last line we have used~\Cref{lem:wave-op-bddness}.  Since this is compatible with the bounds for $E_{0,V}^3$, this completes the proof for the transmitted wave.
    
    We now show how to handle the reflected and incident waves.  We can use~\eqref{eqn:fwd-wave-IBP} for the incident wave and
    \begin{equation*}
        (x-2it\partial_x) e^{i(-kx + tk^2)} = (x-2tk) e^{i(-kx + tk^2)} = i\partial_k e^{i(-kx + tk^2)}
    \end{equation*}
    for the reflected wave to integrate by parts and obtain the bounds
    \begin{equation}\label{eqn:J0-JV-RI-split}\begin{split}
        J_0 (\DFT^{*}_{+,R} + \DFT^{*}_{+,I}) \tilde{w} =& (-\DFT_{+,R}^{-1} \DFT + \DFT_{+,I}^{-1} \DFT) w \\
        & - \frac{i}{\sqrt{2\pi}} \int_{\bbR^+} \partial_k (R(k) \chi_-(x) m_-(x,-k)) e^{i(-kx + tk^2)} \tilde{\phi}(k)\;dk\\
        & + \frac{i}{\sqrt{2\pi}} \int_{\bbR^+} \partial_k (\chi_-(x) m_-(x,k)) e^{i(kx + tk^2)} \tilde{\phi}(k)\;dk\\
        &+ \frac{2it}{\sqrt{2\pi}} \int_{\bbR^+} R(k) \partial_x (\chi_-(x) m_-(x,-k)) e^{i(-kx + tk^2)} \tilde{\phi}(k)\;dk\\
        &- \frac{2it}{\sqrt{2\pi}} \int_{\bbR^+} \partial_x (\chi_-(x) m_-(x,k)) e^{i(kx + tk^2)} \tilde{\phi}(k)\;dk\\
        =& (-\DFT_{+,R}^{-1} \DFT + \DFT_{+,I} \DFT) J_V w + i (-\rmI_{0,V,R} + \rmI_{0,V,I}) - 2i (-\rmII_{0,V,R} + \rmII_{0,V,I})
    \end{split}\end{equation}
    The initial term $(-\DFT_{+,R}^{-1} \DFT + \DFT_{+,I}^{-1} \DFT) J_V w$ combined with the term $\DFT_{+,T} \DFT J_V w$ obtained earlier completes the expression~\eqref{eqn:T-0V-expr} for $T_{0,V}$.  Moreover, since the arguments for $\rmI_{0,V,T}$ and $\rmII_{0,V,T}$ did not use any special property of $T(k)$ beyond~\Cref{lem:TR-pdo-bd,lem:TR-deriv-bdds}, they are easily adapted to when the waves are incoming or reflected.  The $L^2$ boundedness of $T_{0,V}$ follows immediately from~\Cref{lem:pDFT-bddness}.
\end{proof}

\subsection{Control of \texorpdfstring{$J_V$}{JV} in terms of \texorpdfstring{$J_0$}{J0}}\label{sec:JV-to-J0}
\begin{proof}[Proof of~\Cref{thm:JV-to-J0}]
    For $k > 0$, we have that
    \begin{equation*}\begin{split}
        \DFT [J_V w](k)  =& (i\partial_k + 2tk) \DFT w(k)\\
                    =& \frac{i\partial_k + 2tk}{\sqrt{2\pi}} \int \overline{T(k)} \overline{\chi_+(x) m_+(x,k)} e^{-ikx} u(x)\;dx\\
                    &+ \frac{i\partial_k + 2tk}{\sqrt{2\pi}} \int \overline{\chi_-(x) m_-(x,k)} e^{-ikx} u(x)\;dx\\
                    &+ \frac{i\partial_k + 2tk}{\sqrt{2\pi}} \int \overline{R(k)} \overline{\chi_-(x) m_-(x,-k)} e^{ikx} u(x)\;dx
    \end{split}\end{equation*}
    with a similar expansion for $k < 0$.  Using the identity
    \begin{equation*}
        (i\partial_k + 2tk) e^{\mp ikx} = \pm (x + 2it\partial_x) e^{\mp ikx}
    \end{equation*}
    and integrating by parts in $x$, we see that this can be rewritten as
    \begin{align*}
        \DFT J_V w  =& (\DFT_{T} + \DFT_I - \DFT_R) [J_0 w] (x)\\
        &- \frac{i}{\sqrt{2\pi}} \int \partial_k \left(\overline{T(k)} \overline{\chi_+(x) m_+(x,k)}\right)  e^{-ikx} w(x)\;dx\\
        &- \frac{i}{\sqrt{2\pi}} \int \overline{\chi_-(x) \partial_k m_-(x,k)} e^{-ikx} w(x)\;dx\\
        &- \frac{i}{\sqrt{2\pi}} \int \partial_k \left(\overline{R(k)} \overline{\chi_-(x) m_-(x,-k)}\right) e^{ikx} w(x)\;dx\\
        &- \frac{2it}{\sqrt{2\pi}} \int \partial_x \left(\overline{T(k)} \overline{\chi_+(x) m_+(x,k)}\right)  e^{-ikx} w(x)\;dx\\
        &- \frac{2it}{\sqrt{2\pi}} \int \partial_x\left(\overline{\chi_-(x) m_-(x,k)}\right) e^{-ikx} w(x)\;dx\\
        &+ \frac{2it}{\sqrt{2\pi}} \int \partial_x \left(\overline{R(k)} \overline{\chi_-(x) m_-(x,-k)}\right) e^{ikx} w(x)\;dx\\
        =:& (\DFT_{T} + \DFT_I - \DFT_R) [J_0 w] (x) -i  (\rmI_{V,0,T} + \rmI_{V,0,I} + \rmI_{V,0,R}) - 2i (\rmII_{V,0,T} + \rmII_{V,0,I} + \rmII_{V,0,R})
    \end{align*}
    The first term is simply $T_{V,0} J_0$.  By~\Cref{lem:pDFT-bddness,lem:wave-op-bddness}, we have that
    \begin{equation*}
        T_{V,0} = \Omega \cF^*(\DFT_{T} + \DFT_{I} - \DFT_R): L^p \to L^p
    \end{equation*}
    for $p \in (1,2]$.  Thus, it is enough to prove bounds for $\rmI_{V,0,A}$ and $\rmII_{V,0,A}$ ($A \in \{I,T,R\}$) that are compatible with~\eqref{eqn:JV-to-J0-errs}.  Again, we will focus on the bounds for the transmitted waves, as the reflected and incident waves can be controlled using similar arguments.  The term $\rmI_{V,0,T}$ can be seen to be dual to $\rmI_{0,V,T}$, which immediately implies the bound
    \begin{equation*}
        \lVert \rmI_{V,0,T}^{(1)}\rVert_{L^2} \lesssim \lVert w \rVert_{L^2}
    \end{equation*}
    and allows us to absorb this term into $E^1_{V,0}(t)$.  Turning to $\rmII_{V,0,T}$, we define
    \begin{equation*}\begin{split}
        t\rmII_{V,0,T} =& \frac{tT(k)}{\sqrt{2\pi}} \int \overline{\partial_x \chi_+(x) m_+(x,k)} e^{-ikx} w(x)\;dx\\
        &+ \frac{tT(k)}{\sqrt{2\pi}} \int  \overline{\chi_+(x) \partial_x m_+(x,k)} e^{-ikx} w(x)\;dx\\
        =:& t\rmII^{(1)}_{V,0,T} + t\rmII_{V,0,T}^{(2)}
    \end{split}
    \end{equation*}
    The first term $\rmII^{(1)}_{V,0,T}$ is easy to control since $\partial_x \chi_+$ can absorb any space weights, so we turn to $\rmII_{V,0,T}^{(2)}$ and perform a second decomposition using~\eqref{eqn:dx-m-ident}:
    \begin{equation*}\begin{split}
        t\rmII_{V,0,T}^{(2)} =& \frac{t\overline{T(k)}}{\sqrt{2\pi}} \int \int_{x}^\infty \chi_+(x) e^{-ik(x-2y)} V(y) \overline{(m_+(y,k) - 1)}\;dy w(x)\;dx\\
        &+ \frac{t\overline{T(k)}}{\sqrt{2\pi}} \int \int_{x}^\infty \chi_+(x) e^{-ik(x-2y)} V(y) \;dy w(x)\;dx\\
        =:& t\rmII_{V,0,T}^{(2,1)} + t\rmII_{V,0,T}^{(2,2)}
    \end{split}\end{equation*}
    Duality with $\rmII_{0,V,T}^{(2,3)}$ shows that $\rmII_{V,0,T}^{(2,2)}$ obeys the bounds required of $E^3_{V,0}$.  For $\rmII_{V,0,T}^{(2,1)}$, a direct application of~\eqref{eqn:double-V-trick} yields
    \begin{equation*}\begin{split}
        \lVert \rmII_{V,0,T}^{(2,1)} \rVert_{L^2_k} \lesssim& t \int |\jBra{x}^{1-2\gamma} w(x)|\;dx\\
        \lesssim& t \lVert \jBra{x}^{3/2 + a -2\gamma} w(x) \rVert_{L^2_x} 
    \end{split}\end{equation*}
    which is consistent with the required bounds for $E^2_{V,0}$.  Similar estimates hold for the reflected and transmitted waves, completing the proof of~\Cref{thm:JV-to-J0}.
\end{proof}

\section{Leading order asymptotics}\label{sec:asympt}
Now, we show how to obtain the leading-order asymptotics for $u(x,t)$ as $t\to\infty$.  To do this, we define the (free) wave packet with velocity $v$ as
\begin{equation}\label{eqn:wave-packet-def}
    \Psi_v(x,t) = e^{-i\frac{x^2}{4t}} \chi\left(\frac{x - vt}{\sqrt{t}}\right)
\end{equation}
where $\chi$ is an even function supported on $(-10,10)$ with $\int \chi(z)\;dz = 1$.  A straightforward calculation shows that $\Psi_v$ is an approximate solution to the free Schr\"odinger equation:
\begin{equation}\label{eqn:wave-packet-deriv}
    (i\partial_t - \Delta) \Psi_v = \frac{e^{-i\frac{x^2}{4t}}}{2t}\partial_x \left(t^{1/2} \chi'\left(\frac{x-vt}{\sqrt{t}}\right) + i (x-vt) \chi\left(\frac{x-vt}{\sqrt{t}}\right)\right)
\end{equation}
With the wave-packet in hand, we now define the asymptotic profile
\begin{equation}\label{eqn:alpha-eqn}
    \alpha(v,t) = \int u(x,t) \overline{\Psi_v(x,t)}\;dx
\end{equation}

\subsection{Reduction to wave packet dynamics}

The asymptotic profile $\alpha$ is closely related to the pointwise behavior of $u$ in physical and distorted Fourier space (cf.~\cite[Lemma 2.2]{ifrimGlobalBoundsCubic2015}):
\begin{lemma}\label{lem:alpha-bounds}
    We have that
    \begin{equation}\label{eqn:alpha-conv-bounds}
        \lVert \alpha(v,t) \rVert_{L^\infty_v} \lesssim t^{1/2} \lVert u(x,t) \rVert_{L^\infty_x},\quad \lVert \alpha(v,t) \rVert_{L^2_v} \lesssim \lVert u(x,t) \rVert_{L^2_x}
    \end{equation}
    Moreover, defining $R_t = \{v: v > 100 t^{-1/2}\}$, we have the physical space bounds
    \begin{equation}\label{eqn:alpha-phys-bdds}\begin{split}
        \lVert u(vt,t) - t^{-1/2}e^{-i\frac{x^2}{4t}} \alpha(v,t)\rVert_{L^\infty_v(R_t)} \lesssim& t^{-3/4} \left(\lVert J_V u \rVert_{L^2} + \lVert u \rVert_{L^2}\right)\\
        \lVert u(vt,t) - t^{-1/2}e^{-i\frac{x^2}{4t}} \alpha(v,t)\rVert_{L^2_v(R_t)} \lesssim& t^{-1} \left(\lVert J_V u \rVert_{L^2} + \lVert u \rVert_{L^2}\right)
    \end{split}\end{equation}
    In addition, $\alpha$ and is a good approximation to $\tilde{u}$ in the sense that
    \begin{equation}\label{eqn:alpha-DFT-bdds}\begin{split}
        \lVert e^{-it\frac{v^2}{4}} \tilde{u}(-\frac{v}{2},t) - \alpha(v,t) \rVert_{L^\infty_v(R_t)} \lesssim t^{-1/4} \lVert J_V u \rVert_{L^2} + t^{1 - \gamma/2} \lVert u \rVert_{L^2}\\
        \lVert e^{-it\frac{v^2}{4}} \tilde{u}(-\frac{v}{2},t) - \alpha(v,t) \rVert_{L^2_v(R_t)} \lesssim t^{-1/2} \lVert J_V u \rVert_{L^2} + t^{3/4 - \gamma/2} \lVert u \rVert_{L^2}
    \end{split}\end{equation}
\end{lemma}
\begin{proof}    
    Let us define $q = e^{i\frac{x^2}{4t}} u$.  Then
    \begin{equation}\label{eqn:alpha-conv-form}\begin{split}
        t^{-1/2}\alpha(v,t) =& t^{-1/2}\int q(x,t) \chi\left(\frac{x-vt}{\sqrt{t}}\right) \;dx\\
                    =& \int q(tz,t) t^{1/2} \chi(t^{1/2} (z - v)) \;dz\\
                    =& [q(t\cdot,t) * t^{1/2} \chi(t^{1/2}\cdot)](v)
    \end{split}\end{equation}
    The first two inequalities in~\eqref{eqn:alpha-conv-bounds} thus follow from Young's inequality.  
    Turning to the bounds~\eqref{eqn:alpha-phys-bdds}, we note that it is equivalent to prove $L^\infty$ and $L^2$ bounds for $q(vt, t) - t^{1/2} \alpha(v,t)$.  Using the fact that $\chi$ has integral $1$, we see that
    \begin{equation}\label{eqn:q-alpha-diff}
        q(vt, t) - t^{-1/2} \alpha(v,t) = \int (q(vt,t) - q((v-z)t,t)) t^{1/2} \chi(z t^{1/2}) \;dz
    \end{equation}
    Let us express the difference as an integral:
    \begin{equation}\label{eqn:q-diff-int}
        q(vt,t) - q((v-z)t,t) = z \int_0^1 t\partial_x q((v-hz)t,t)\;dh
    \end{equation}
    Now, noting that
    \begin{equation}\label{eqn:d-x-q-identity}
        -2it\partial_x q(x,t) = e^{-i\frac{x^2}{4t}} J_0 u(x,t)
    \end{equation}
    and using~\Cref{thm:J0-to-JV}, we see that
    \begin{equation}\label{eqn:d-x-q-to-JV}\begin{split}
        t\partial_x q    =& \frac{ie^{-i\frac{x^2}{4t}}}{2} J_0 u\\
                        =& \frac{ie^{-i\frac{x^2}{4t}}}{2} (T_{0,V} J_V u + E^1_{0,V}(t) u + E^2_{0,V}(t) u + E^3_{0,V}(t) u)\\
    \end{split}\end{equation}
    For $v \in R^+_t$, $z \in \supp(\chi(t^{1/2}\cdot)$, we have $|v - z| \gtrsim |v|$, so
    \begin{equation}\label{eqn:q-diff-infty}\begin{split}
        \left| q(vt,t) - q((v-z)t,t) \right| \lesssim& \frac{|z|^{1/2}}{t^{1/2}} \left(\lVert T_{0,V} J_V u \rVert_{L^2} + \lVert E^1_{0,V} u \rVert_{L^2}+ |vt|^{\frac{3}{2}+a-2\gamma} \lVert \jBra{x}^{2\gamma - a - \frac{3}{2}}E^2_{0,V} u \rVert_{L^2} \right.\\
        &\left.\phantom{\frac{|z|^{1/2}}{t^{1/2}} } + |vt|^{-\gamma}  \lVert \jBra{x}^\gamma E^3_{0,V} u \rVert_{L^2}\right)\\
        \lesssim& \frac{|z|^{1/2}}{t^{1/2}}\left(1 + |vt|^{3/2+a-2\gamma} t^{1/4} + |vt|^{-\gamma} t\right) \left(\lVert J_V u \rVert_{L^2} + \lVert u \rVert_{L^2}\right)\\
        \lesssim& \frac{|z|^{1/2}}{t^{1/2}}\left(\lVert J_V u \rVert_{L^2} + \lVert u \rVert_{L^2}\right)
    \end{split}\end{equation}
    where the simplification on the last line follows from the fact that $|vt| \gtrsim t^{1/2}$.  Inserting this into~\eqref{eqn:q-alpha-diff} yields
    \begin{equation*}
        |q(vt,t) - t^{-1/2} \alpha(v,t)| \lesssim  \int |z|^{1/2} \chi(z t^{1/2})\;dz \left(\lVert J_V u \rVert_{L^2} + \lVert u \rVert_{L^2}\right) \lesssim t^{-3/4} \left(\lVert J_V u \rVert_{L^2} + \lVert u \rVert_{L^2}\right)
    \end{equation*}
    for $v \in R_t$, which gives the $L^\infty_v$ bound in~\eqref{eqn:alpha-phys-bdds}.  For the $L^2_v$ bound, we note that since $|v-z| \gtrsim |v|$,
    \begin{equation}\label{eqn:q-diff}\begin{split}
        \left \lVert q(vt,t) - q((v-z)t,t) \right\rVert_{L^2_v} 
            \lesssim& |z| \int_0^1 \lVert J_0 u((v-hz)t,t) \rVert_{L^2_v(R_t)} \;dh\\
            \lesssim& |z| \lVert T_{0,V} J_V u((v-hz)t,t) \rVert_{L^2_v(R_t)}+ |z| \lVert E^1_{0,V} u((v-hz),t) \rVert_{L^2_v(R_t)}\\
            &+ |z| \lVert E^2_{0,V} u((v-hz),t) \rVert_{L^2_v(R_t)} + |z| \lVert E^3_{0,V} u((v-hz),t) \rVert_{L^2_v(R_t)}\\
            \lesssim& \frac{|z|}{t^{1/2}} \lVert T_{0,V} J_V u(x,t) \rVert_{L^2}+ \frac{|z|}{t^{1/2}} \lVert E^1_{0,V} u(x,t) \rVert_{L^2}\\
            &+ \frac{|z|}{t^{1/2}} \lVert E^2_{0,V} u(x,t) \rVert_{L^2}+ \frac{|z|}{t^{1/2}} \lVert E^3_{0,V} u(x,t) \rVert_{L^2}\\
            \lesssim& \frac{|z|}{t^{1/2}} (1 + t^{1+a/2-\gamma} + t^{1-\gamma/2})(\lVert J_V u \rVert_{L^2} + \lVert u \rVert_{L^2})\\
            \lesssim& \frac{|z|}{t^{1/2}} (\lVert J_V u \rVert_{L^2} + \lVert u \rVert_{L^2})
    \end{split}\end{equation}
    Inserting this last bound into~\eqref{eqn:q-alpha-diff}, we see that
    \begin{equation*}\begin{split}
        \lVert q(vt, t) - t^{-1/2} \alpha(v,t) \rVert_{L^2_v(R_t)} \leq& \int \lVert q(vt,t) - q((v-z)t,t) \rVert_{L^2_v(R_t)} t^{1/2} \chi(z t^{1/2}) \;dz \\
        \lesssim& \int |z| \chi(z t^{1/2})\;dz (\lVert J_V u \rVert_{L^2} + \lVert u \rVert_{L^2}) \lesssim t^{-1} (\lVert J_V u \rVert_{L^2} + \lVert u \rVert_{L^2})
    \end{split}\end{equation*}
    which is the second inequality in~\eqref{eqn:alpha-phys-bdds}.
    
    We now prove the estimates~\eqref{eqn:alpha-DFT-bdds} in distorted Fourier space.  Here, we will write $R_t = R^+_t \cup R^-_t$, where $R^\pm_t = \{v \in R_t : \pm v > 0\}$.  We will focus on the estimates on $R^+_t$: the estimates on $R^-_t$ are similar once we make the usual modifications.  To begin, we observe that the flat Fourier transform of $\Psi_v$ is given by
    \begin{equation}\begin{split}
        \mathcal{F} \Psi_v(\xi,t) =& \frac{e^{it\xi^2} e^{it(\frac{v}{2}+\xi)^2}}{\sqrt{2\pi}} \int e^{-i(x-vt)(\frac{v}{2}+\xi)} e^{-i\frac{(x-vt)^2}{4t}} \chi\left(\frac{x-vt}{\sqrt{t}}\right)\;dx \\
        =& e^{it\xi^2} t^{1/2} \chi_1(t^{1/2}(\frac{v}{2}+\xi))
    \end{split}\end{equation}
    where
    \begin{equation*}
        \chi_1(\xi) = e^{-i\xi^2} \mathcal{F}\left\{ e^{-i\frac{x^2}{4}} \chi(x)\right\}(\xi)
    \end{equation*}
    has the property that
    \begin{equation*}\begin{split}
        \int \chi_1(\xi)\;d\xi  =& \int e^{-i\xi^2} \mathcal{F}\left\{ e^{-i\frac{x^2}{4}} \chi(x)\right\}(\xi)\;d\xi\\
                                =& \int \mathcal{F}^{-1}\left\{e^{-i\xi^2}\right\}(x) e^{-i\frac{x^2}{4}} \chi(x)\;dx\\
                                =& \int \chi(x)\;dx = 1
    \end{split}\end{equation*}
    Of course, to prove~\eqref{eqn:alpha-DFT-bdds}, the \textit{distorted} Fourier transform of $\Psi_v$ is more relevant.  As usual, we will consider the cases $k > 0$ and $k < 0$ separately.  Starting with $k > 0$, we observe that for $v \in R^+_t$, $\chi(t^{-1/2}(x - vt))$ vanishes on $x < 1$, so only the transmitted wave contributes to the DFT:
    \begin{equation*}\begin{split}
        \DFT(\Psi_v)(k,t)   =& \frac{1}{\sqrt{2\pi}}\int \overline{T(k) \chi_+(x) m_+(x,k) e^{ixk}} e^{-i\frac{x^2}{4t}} \chi\left(\frac{x-vt}{\sqrt{t}}\right)\;dx\\
                            =& \overline{T(k)} \mathcal{F}(\Psi_v)(k,t) + \overline{T(k)} \int \overline{(m_+(x,k) - 1)}e^{-ixk} \chi\left(\frac{x-vt}{\sqrt{t}}\right)\;dx
    \end{split}\end{equation*}
    A quick calculation shows that
    \begin{equation*}
        \left|\overline{T(k)} \int \overline{(m_+(x,k) - 1)}e^{-ixk} \chi\left(\frac{x-vt}{\sqrt{t}}\right)\;dx\right| \lesssim \frac{|T(k)|}{\jBra{k}} \int \jBra{x}^{1-\gamma} \chi\left(\frac{x-vt}{\sqrt{t}}\right)\;dx \lesssim t^{1/2}\jBra{k}^{-1}\jBra{vt}^{1-\gamma}
    \end{equation*}
    so
    \begin{equation*}\begin{split}
        \DFT(\Psi_v)(k,t)  =& \overline{T(k)} \mathcal{F}(\Psi_v)(k,t) + O(\jBra{k}^{-1} t^{1/2}\jBra{vt}^{1-\gamma})
    \end{split}\end{equation*}
    A similar calculation shows that for $k < 0$
    \begin{equation*}\begin{split}
        \DFT(\Psi_v)(k,t)   =& \frac{1}{\sqrt{2\pi}}\int \left(\overline{R(-k) \chi_+(x) m_+(x,k) e^{-ixk} + \chi_+(x) m_+(x,-k)e^{ixk}}\right) e^{-i\frac{x^2}{4t}} \chi\left(\frac{x-vt}{\sqrt{t}}\right)\;dx\\
                            =& \overline{R(-k)} \mathcal{F}(\Psi_v)(-k,t) + \mathcal{F}(\Psi_v)(k,t) + O(\jBra{k}^{-1} t^{1/2}\jBra{vt}^{1-\gamma})
    \end{split}\end{equation*}
    Thus, using the Plancherel theorem for the distorted Fourier transform, we find that
    \begin{align}
        \alpha(v,t)     =& \int \tilde{u}(k,t) \overline{\DFT \Psi_v(k,t))}\;dk\notag\\
        \begin{split}
                        =&  t^\frac{1}{2}\int_{k>0}T(k) \tilde{f}(k,t) \chi_1\left(t^\frac{1}{2}\left(\frac{v}{2}+k\right)\right)\;dk + t^\frac{1}{2}\int_{k<0} R(k) \tilde{f}(k,t) \chi_1\left(t^\frac{1}{2}\left(\frac{v}{2}-k\right)\right)\;dk\\
                        &+ t^\frac{1}{2}\int_{k<0} \tilde{f}(k,t) \chi_1\left(t^\frac{1}{2}\left(\frac{v}{2}+k\right)\right)\;dk + \int_{\bbR} \tilde{u}(k,t) O(\jBra{k}^{-1}t^\frac{1}{2}\jBra{vt}^{1-\gamma})\;dk
    \label{eqn:alpha-DFT-integrals}\end{split}\end{align}
    Focusing on the integral for the transmitted wave, we have that
    \begin{equation}\label{eqn:alpha-DFT-trans-wave}\begin{split}
        \left|t^{\frac{1}{2}}\int_{k>0} T(k) \tilde{f}(k,t) \chi_1(t^{\frac{1}{2}}(v/2+k))\;dk\right| \lesssim& \left\lVert \frac{\tilde{f}(k,t)}{k} \right\rVert_{L^2_k} \lVert t^{\frac{1}{2}}(v/2+k)\chi_1(t^{\frac{1}{2}}(v/2+k)) \rVert_{L^2_k}\\
        \lesssim_N& t^{-1/4}|t^\frac{1}{2} v|^{-N} \lVert J_V u \rVert_{L^2} 
    \end{split}\end{equation}
    where on the last line we have used the fact that $\chi_1$ is in the Schwartz class and therefore decays rapidly.  A similar estimate holds for the reflected wave, and applying Cauchy-Schwarz shows that the last term in~\eqref{eqn:alpha-DFT-integrals} is $O(t^{3/2-\gamma}|v|^{1-\gamma})\lVert u \rVert_{L^2}$.  Hence, the main contribution to $\alpha$ comes from the incident wave:
    \begin{equation*}
        \alpha(v,t) = t^{1/2}\int_{k<0} \tilde{f}(k,t) \chi_1(t^{1/2}(\frac{v}{2}+k))\;dk + O\left(t^{-1/4} \lVert J_v u \rVert_{L^2} + t^{3/2-\gamma}|v|^{1-\gamma}\lVert u \rVert_{L^2}\right)
    \end{equation*}
    Since the bound~\eqref{eqn:alpha-DFT-trans-wave} did not depend on any special properties of $T(k)$, we have that
    \begin{equation*}
        \left|t^{1/2}\int_{k<0} \tilde{f}(k,t) \chi_1(t^{1/2}(\frac{v}{2}+k))\;dk\right| \lesssim_N t^{-1/4}|t^\frac{1}{2} v|^{-N} \lVert J_V u \rVert_{L^2} 
    \end{equation*}
    Thus, noting that $t^{1/2} \chi_1(t^{1/2}\cdot)$ has integral $1$, we can argue as in the proof of~\eqref{eqn:q-diff} to find that
    \begin{equation*}\begin{split}
        t^{1/2}\int_{k<0} \tilde{f}(k,t) \chi_1(t^{1/2}(\frac{v}{2}+k))\;dk - \tilde{f}\left(-\frac{v}{2},t\right) =& t^{1/2} \int_{\bbR} \tilde{f}\left(\kappa - \frac{v}{2},t\right) \chi_1(t^{1/2}\kappa)\;d\kappa - \tilde{f}\left(-\frac{v}{2},t\right)\\
        &+ O\left(t^{-1/4}|t^\frac{1}{2} v|^{-N} \lVert J_V u \rVert_{L^2}  \right)\\
        =& t^{1/2} \int_{\bbR} \left(\tilde{f}\left(\kappa - \frac{v}{2},t\right) - \tilde{f}\left(-\frac{v}{2},t\right)\right) \chi_1(t^{1/2}\kappa)\;d\kappa\\
        &+ O\left(t^{-1/4}|t^\frac{1}{2} v|^{-N} \lVert J_V u \rVert_{L^2}  \right)\\
        =& \int \int_0^1 \partial_k \tilde{f}\left(h\kappa - \frac{v}{2},t\right) t^{1/2} \kappa \chi_1(t^{1/2}\kappa)\;dh d\kappa\\
        &+ O\left(t^{-1/4}|t^\frac{1}{2} v|^{-N} \lVert J_V u \rVert_{L^2}  \right)\\
        =& O\left(t^{-1/4} (1+ |t^\frac{1}{2} v|^{-N}) \lVert J_V u \rVert_{L^2}  \right)\\
    \end{split}\end{equation*}
    so
    \begin{equation*}
        |\tilde{f}(-\frac{v}{2},t) - \alpha(v,t)| \lesssim t^{-1/4} \lVert J_V u \rVert_{L^2} + t^{3/2-\gamma}|v|^{1-\gamma} \lVert u \rVert_{L^2}
    \end{equation*}
    which gives the $L^\infty$ part of~\eqref{eqn:alpha-DFT-bdds}.  For the $L^2$ estimate, we note that
    \begin{equation*}\begin{split}
        \left\lVert \tilde{f}(-v/2, t) - t^{\frac{1}{2}}\int_\bbR \tilde{f}(k,t) \chi_1(t^{\frac{1}{2}}(v/2+k))\;dk \right\rVert_{L^2_v(R^+_t)}        
        \leq& \int_\bbR \lVert \partial_v \tilde{f} \rVert_{L^2_v} |t^{1/2}\kappa \chi_1(t^{1/2}\kappa)|\;d\kappa\\
        \lesssim& t^{-1/2}\lVert J_V u \rVert_{L^2} 
    \end{split}\end{equation*}
    Combining this with the estimates for the other terms, we find that
    \begin{equation*}
        \lVert \tilde{f}(-\frac{v}{2},t) - \alpha(v,t) \rVert_{L^2(R_t)} \lesssim t^{-1/2} \lVert J_V u \rVert_{L^2} + t^{3/4 - \gamma/2} \lVert u \rVert_{L^2}
    \end{equation*}
    which completes the proof of~\eqref{eqn:alpha-DFT-bdds}.
\end{proof}

\subsection{Asymptotics for \texorpdfstring{$\alpha$}{alpha}}

Now, we show that the evolution of $\alpha(v,t)$ is essentially given by a Hamiltonian ODE for $|v| \gtrsim t^{-1/2}$, which can be integrated to give the dynamics.
\begin{lemma}\label{lem:alpha-dynamics-lem}
    For $|v| > 100 t^{-1/2}$, we have that
    \begin{equation*}
        \partial_t \alpha(v,t) = \mp \frac{i}{t} |\alpha(v,t)|^2 \alpha(v,t) + R(v,t)
    \end{equation*}
    where
    \begin{equation}\label{eqn:alpha-rem-bdds}\begin{split}
        \lVert R(v,t) \rVert_{L^\infty(R_t)} \lesssim& t^{-5/4} \left(\lVert J_V u \rVert_{L^2} + \lVert u \rVert_{L^2}\right)\\
        \lVert R(v,t) \rVert_{L^2(R_t)} \lesssim& t^{-3/2} \left(\lVert J_V u \rVert_{L^2} + \lVert u \rVert_{L^2}\right)
    \end{split}\end{equation}
\end{lemma}
\begin{proof}
    We have that
    \begin{subequations}
    \begin{align}
        \partial_t \alpha(v,t)  =& \int \partial_t u(x,t) \overline{\Psi_v(x,t)} + u(x,t) \overline{\partial_t \Psi_v(x,t)}\;dx\notag\\
                                =& \mp i\int |u|^2u(x,t) \overline{\Psi_v(x,t)}\;dx\label{eqn:dt-alpha-nonlin}\\
                                &+ \int u(x,t)\overline{(\partial_t + i\Delta)\Psi_v(x,t)} \;dx\label{eqn:dt-alpha-lin-err}\\
                                &+ \int iV u(x,t) \overline{\Psi_v(x,t)}\;dx\label{eqn:dt-alpha-pot-err}
    \end{align}\end{subequations}
    We first consider the terms~\cref{eqn:dt-alpha-lin-err,eqn:dt-alpha-pot-err}, which will be subsumed into the error term $R$, and then show how to extract the leading-order asymptotics from~\eqref{eqn:dt-alpha-nonlin}.  By the assumption that $|v| > 100 t^{-1/2}$ together with the improved local decay estimate~\Cref{lem:improved-local-decay}, we see that
    \begin{equation*}
        |\eqref{eqn:dt-alpha-pot-err}| \lesssim \lVert \jBra{x}^{-1}u \rVert_{L^\infty} \int |\jBra{x} V(x)| \chi(t^{-1/2}(x-vt))\;dx \lesssim t^{-3/4} |vt|^{1-\gamma} \lVert J_V u \rVert_{L^2}
    \end{equation*}
    which is consistent with the decay we require for $R$.  Turning to~\eqref{eqn:dt-alpha-lin-err} and using~\eqref{eqn:wave-packet-deriv}, we see that
    \begin{equation*}\begin{split}
        \left|\eqref{eqn:dt-alpha-lin-err}\right| =& \left|\frac{1}{2t}\int \partial_x \left(e^{it\frac{x^2}{4}} u(x,t)\right) \left(t^{\frac{1}{2}} \chi'\left(\frac{x-vt}{\sqrt{t}}\right) - i (x-vt) \chi\left(\frac{x-vt}{\sqrt{t}}\right)\right)\;dx\right|\\
        =& \left| \frac{1}{4t^2} \int e^{-i\frac{x^2}{4t}} J_0 u(x,t) \left(t^{\frac{1}{2}} \chi'\left(\frac{x-vt}{\sqrt{t}}\right) - i (x-vt) \chi\left(\frac{x-vt}{\sqrt{t}}\right)\right)\;dx\right|\\
    \end{split}
    \end{equation*}
    Using~\Cref{thm:J0-to-JV} and arguing as in~\eqref{eqn:q-diff}, we see that
    \begin{equation*}\begin{split}
        \left|\eqref{eqn:dt-alpha-lin-err}\right| \lesssim& t^{-5/4} \left(\lVert J_V u \rVert_{L^2} + \lVert u \rVert_{L^2}\right) 
    \end{split}
    \end{equation*}
    Moreover, noting that this expression is a convolution, we can apply Young's inequality to find that
    \begin{equation*}
        \left\lVert \eqref{eqn:dt-alpha-lin-err} \right\rVert_{L^2(R_t)} \lesssim t^{-3/2} \left(\lVert J_V u \rVert_{L^2} + \lVert u \rVert_{L^2}\right)
    \end{equation*}
    Now, let us consider the nonlinear term~\eqref{eqn:dt-alpha-nonlin}.  We have
    \begin{equation}\begin{split}
        \eqref{eqn:dt-alpha-nonlin} =& \mp\frac{i}{t}\alpha(v,t)|\alpha(v,t)|^2 \\
        &\mp i\alpha(v,t) (|u(vt, t)|^2 - \frac{1}{t}|\alpha(v,t)|^2)\\
        &\mp i\int (|u(x,t)|^2 - |u(vt, t)|^2)u(x,t) \overline{\Psi_v(x,t)}\;dx \\
        =& \mp\frac{i}{t}\alpha(v,t)|\alpha(v,t)|^2 + R_1 + R_2
    \end{split}
    \end{equation}
    so it is enough to show that the remainder terms obey the estimates~\eqref{eqn:alpha-rem-bdds}.  For $R_1$, a simple application of~\eqref{eqn:alpha-conv-bounds} and~\eqref{eqn:alpha-phys-bdds} shows that
    \begin{equation*}\begin{split}
        \lVert R_1 \rVert_{L^\infty(R_t)} \lesssim& \lVert \alpha \rVert_{L^\infty(R_t)}\left(t^{-1/2}\lVert \alpha \rVert_{L^\infty(R_t)} + \lVert u \rVert_{L^\infty(R_t)}\right) \lVert |u(vt,t)| - t^{-1/2}|\alpha(v,t)| \rVert_{L^\infty(R_t)}\\
        \lesssim& M^2\epsilon^2 t^{-5/4} \left(\lVert J_V u \rVert_{L^2} + \lVert u \rVert_{L^2}\right)
    \end{split}\end{equation*}
    and
    \begin{equation*}\begin{split}
        \lVert R_1 \rVert_{L^2(R_t)} \lesssim& \lVert \alpha \rVert_{L^\infty(R_t)}\left(t^{-1/2}\lVert \alpha \rVert_{L^\infty(R_t)} + \lVert u \rVert_{L^\infty(R_t)}\right) \lVert |u(vt,t)| - t^{-1/2}|\alpha(v,t)| \rVert_{L^2(R_t)}\\
        \lesssim& M^2\epsilon^2 t^{-3/2} \left(\lVert J_V u \rVert_{L^2} + \lVert u \rVert_{L^2}\right)
    \end{split}\end{equation*}
    Which is consistent with~\eqref{eqn:alpha-rem-bdds} since $M \ll \epsilon^{-2/3}$.  Turning to $R_2$, by using the fact that $|u| = |q|$, we see that
    \begin{equation*}
        |R_2(v,t)| \lesssim M^2 \epsilon^2 t^{-1/2} \int |\chi(t^{1/2} z)| |q((v-z)t, t) - q(vt,t)| t^{1/2}\;dz
    \end{equation*}
    so, applying the same arguments as in the proof of~\eqref{eqn:alpha-phys-bdds}, we see that $R_2$ satisfies~\eqref{eqn:alpha-rem-bdds}. %
\end{proof}

Now, we use~\Cref{lem:alpha-dynamics-lem} to obtain precise asymptotics for $\alpha$.
\begin{prop}\label{prop:alpha-asymptotics}
    For $t \geq 1$, we have the estimate
    \begin{equation}\label{eqn:alpha-asympt-bdds}
        \sup_{t \geq 1} \lVert \alpha(v,t) \rVert_{L^\infty(R_t)} \lesssim \epsilon
    \end{equation}
    Moreover, assuming that the bootstrap bounds~\eqref{eqn:bootstrap-hypos} hold globally, the leading order asymptotics of $\alpha$ for $(v,t) \in \{(v,t) : t \geq 1, v \in \overline{R_t}\}$ are given by
    \begin{equation}\label{eqn:alpha-asymptotics}
        \alpha(v,t) = \exp\left(\mp i \int_1^t \frac{i}{s} |\alpha(v,s)|^2\;ds\right) A(v) + \mathfrak{R}(v,t)
    \end{equation}
    where
    \begin{equation}\label{eqn:frR-v-infty-bdd}
        \lVert \mathfrak{R}(v,t) \rVert_{L^\infty_v(R_t)} \lesssim \epsilon t^{-1/4+\delta}
    \end{equation}
    \begin{equation}\label{eqn:frR-v-2-bdd}
        \lVert \mathfrak{R}(v,t) \rVert_{L^2_v(R_t)} \lesssim \epsilon t^{-1/2+\delta}
    \end{equation}
\end{prop}
\begin{proof}
    We begin by proving~\eqref{eqn:alpha-asympt-bdds}.  Define
    \begin{equation*}
        \beta(v,t) = \exp\left(\pm i \int_1^t \frac{i}{s} |\alpha(v,s)|^2\;ds\right)\alpha(v,t)
    \end{equation*}
    It follows from~\Cref{lem:alpha-dynamics-lem} that $\partial_t \beta(v,t) = R(v,t)$, where $R$ obeys~\eqref{eqn:alpha-rem-bdds}.  Moreover, since $|\alpha| = |\beta|$, it suffices to prove that $\beta$ is bounded for $v$ in $R_t$.  Define $t_0(v) = \max(1, 100v^{-2})$, and observe that for $|v| > 10$, $v \in R_{t_0(v)}$, while for $|v| \leq 10$, $v \in \partial R_{t_0(v)}$.  We first show that $\beta(v, t_0(v))$ is bounded pointwise in $v$, and then use~\Cref{lem:alpha-dynamics-lem} to extend this bound to $t > t_0(v)$.  If $|v| \geq 10$, then $t_0(v) = 1$ and
    \begin{equation*}
        |\beta(v, t_0(v)| = |\alpha(v, 1)| \lesssim \lVert u(x,1) \rVert_{L^\infty_x} \lesssim \epsilon
    \end{equation*}
    where the first inequality is~\eqref{eqn:alpha-conv-bounds} and the second comes from using~\Cref{lem:H-1-infty-decay} together with~\eqref{eqn:initial-cond-size}.  If $|v| < 10$, then the wave packet $\Psi_v(x,t_0(v))$ defined in~\eqref{eqn:wave-packet-def} is supported in $|x| < 1000 t_0^{1/2}(v)$ with $\lVert \Psi_v(x,t_0(v)) \rVert_{L^1} = \sqrt{t_0(v)}$.  Thus, using~\Cref{lem:improved-low-x-bdds}, we see that
    \begin{equation}\label{eqn:beta-t-0-size}\begin{split}
        |\beta(v, t_0(v))| =& |\alpha(v, t_0(v))| \\
        \lesssim& \sqrt{t_0(v)}\lVert u(x,t_0(v)) \rVert_{L^\infty(|x|< 1000t_0^{1/2}(v))}\\
        \lesssim& \epsilon [t_0(v)]^{-1/4+\delta}\\
        \ll& \epsilon
    \end{split}\end{equation}
    It follows that~\eqref{eqn:alpha-asympt-bdds} holds for $t = t_0(v)$.  For $t > t_0(v)$, we note that
    \begin{equation}\label{eqn:beta-int-est}\begin{split}
        |\beta(v,t) - \beta(v, t_0(v))| \leq& \int_{t_0(v)}^t |\partial_s \beta(v,s)|\;ds\\
        \lesssim& \int_{t_0(v)}^t O(t^{-5/4} (\lVert J_V u \rVert_{L^2} + \lVert u \rVert_{L^2} ))\;ds\\
        \lesssim& \epsilon
    \end{split}\end{equation}
    which completes the proof of~\eqref{eqn:alpha-asympt-bdds}.

    Now, let us assume~\eqref{eqn:bootstrap-hypos} and prove~\eqref{eqn:alpha-asymptotics}.  By modifying~\eqref{eqn:beta-int-est}, we see that $\beta(v,t)$ is Cauchy as $t \to \infty$, so we can define \begin{equation}
        A(v) = \lim_{t \to \infty} \beta(v,t)
    \end{equation}
    To prove~\eqref{eqn:frR-v-infty-bdd} and~\eqref{eqn:frR-v-2-bdd}, we note that
    \begin{equation*}
        |\frR(v,t)| = |\beta(v,t) - A(v)| \leq \int_t^\infty |\partial_s \beta(v,s)|\;ds = \int_t^\infty |R(v,s)|\;ds
    \end{equation*}
    so the result follows from~\eqref{eqn:alpha-rem-bdds}.
\end{proof}

\section{Proof of \texorpdfstring{\Cref{thm:main-theorem}}{Theorem 1}}\label{sec:main-thm-proof}

We are now in a position to prove~\Cref{thm:main-theorem}:
\begin{proof}[Proof of~\Cref{thm:main-theorem}]
    Let us begin by assuming~\eqref{eqn:bootstrap-hypos}.  By~\Cref{prop:weighted-est}, we have that
    \begin{equation*}
        \sup_{t \in [1,T]}t^{-\delta}\lVert J_V u \rVert_{L^2} \lesssim C\epsilon
    \end{equation*}
    On the other hand, by using~\Cref{lem:alpha-bounds} and~\eqref{eqn:alpha-asympt-bdds}, we have that
    \begin{equation*}\begin{split}
        \lVert u(x,t) \rVert_{L^\infty(|x| \geq 100 t^{1/2})} =& \lVert u(vt,t) \rVert_{L^\infty_v(R_t)}\\
        \leq& t^{-1/2} \lVert \alpha(v,t) \rVert_{L^\infty} + \lVert u(vt,t) - t^{-1/2}e^{-i\frac{x^2}{4t}} \alpha(v,t)\rVert_{L^\infty_v(R_t)}\\
        \lesssim& \epsilon t^{-1/2} + \epsilon t^{-3/4+\delta}
    \end{split}\end{equation*}
    so~\eqref{eqn:bootstrap-up} follows.  A standard continuity argument shows that~\eqref{eqn:bootstrap-up} holds for all time, which together with~\Cref{lem:improved-low-x-bdds} is sufficient to prove~\eqref{eqn:u-decay}.  To prove the asymptotics~\eqref{eqn:u-asympt}, we note that in the region $|x| > 100 t^{1/2}$,
    \begin{equation*}
        u(x,t) = t^{-1/2} \exp\left(i\frac{x^2}{4t}\right) \alpha(x/t,t) + o_{t^{1/2}L^\infty \cap L^2}(1)
    \end{equation*}
    Replacing $\alpha$ by the asymptotic equivalent derived in~\Cref{prop:alpha-asymptotics} and expanding $\alpha$ in terms of $\tilde{u}$ using~\Cref{lem:alpha-bounds}, we find that
    \begin{equation*}\begin{split}
        u(x,t) =& t^{-1/2} \exp\left(i\frac{x^2}{4t} \mp i \int_1^t \frac{i}{s} |\alpha(v,s)|^2 \;ds\right) A(x/t,t) + o_{t^{1/2}L^\infty \cap L^2}(1)\\
        =& t^{-1/2}\exp\left(i\frac{x^2}{4t} \mp i \int_1^t \frac{i}{s} |\tilde{u}(-x/(2t),s)|^2 \;ds\right) A(x/t,t) + o_{t^{1/2}L^\infty \cap L^2}(1)
    \end{split}\end{equation*}
    Taking $u_\infty = A$, this gives~\eqref{eqn:u-asympt} for $|x| > 100 t^{1/2}$.  To obtain the result for $|x| \leq 100 t^{1/2}$, we first notice that for $|v| < 100 t^{-1/2}$, $t < t_0(v)$, so by~\eqref{eqn:beta-t-0-size} and~\eqref{eqn:frR-v-infty-bdd}
    \begin{equation}\begin{split}
        |A(v)|  =& |\alpha(v,t_0(v))| + |\fR(v,t_0(v))|\\
                \lesssim& \epsilon (t_0(v)^{-1/4+\delta} + t_0(v)^{1-\gamma/2})\\
                \lesssim& \epsilon t^{-1/4+\delta}
    \end{split}\end{equation}
    where we have used the definition~\eqref{eqn:wave-packet-def} on the second line and~\Cref{lem:improved-low-x-bdds} on the second to last line.  In particular, for $|x| < 100 t^{1/2}$ the leading order term has the same size as the error term:
    \begin{equation*}
        t^{-1/2}\exp\left( i\frac{x^2}{4t} \mp i \int_1^t \frac{i}{s} |\tilde{u}(-x/(2t),s)|^2\;ds\right) A(x/t,t) = o_{t^{1/2}L^\infty_x \cap L^2_x}(1)
    \end{equation*}
    Since in this region $u(x,t) = o_{t^{1/2}L^\infty_x \cap L^2_x}(1)$ by~\Cref{lem:improved-low-x-bdds}, we trivially have that
    \begin{equation*}
        u(x,t) - t^{-1/2}\exp\left( i\frac{x^2}{4t} \mp i \int_1^t \frac{i}{s}|\tilde{u}(-x/(2t),s)|^2\;ds\right) A(x/t,t) = o_{t^{1/2}L^\infty_x \cap L^2_x}(1)
    \end{equation*}
    which completes the proof.
\end{proof}

\bibliographystyle{plain}
\bibliography{sources}
\end{document}